\documentclass[a4paper,twoside]{article}

\def\shorttitle{Coefficient Inverse Problems for General Hyperbolic Equations}
\def\shortauthor{J. Yu, Y. Liu and M. Yamamoto}

\usepackage{amsfonts}
\usepackage{amsmath}
\usepackage{amssymb}
\usepackage{amscd}
\usepackage{amsbsy}
\usepackage{amsthm}
\usepackage{bm}
\usepackage{empheq}
\usepackage{graphicx}
\usepackage{psfrag}
\usepackage{color}
\usepackage{cite}
\usepackage{multicol}
\usepackage{subfigure}
\usepackage[normalem]{ulem}

\allowdisplaybreaks[4]

\newfont{\myfnt}{cmssi10 scaled 1440}
\numberwithin{equation}{section}
\catcode`@=11
\def\ps@nk{\def\@oddhead{\vbox{\hbox to \hsize{\pic \footnotesize \it \shorttitle
\hfill \rm \thepage} \vspace{1mm} \vspace*{-2mm}}}
\def\@evenhead{\vbox{\hbox to \hsize{\pic \footnotesize \rm \thepage \hfill \it \shortauthor}
\vspace{1mm} \vspace*{-2mm}}}
\def\@oddfoot{} \def\@evenfoot{}}
\def\ps@first{\def\@oddhead{\vbox{\hbox to \hsize{\pic \footnotesize
} \break}}
\def\@oddfoot{} \def\@evenfoot{}}
\newtheoremstyle{thmstyle}
  {6pt}
  {6pt}
  {\it}
  {}
  {\bf}
  {}
  {.5em}
  {}
\newtheoremstyle{remstyle}
  {6pt}
  {6pt}
  {\rm}
  {}
  {\bf}
  {}
  {.5em}
  {}
\def\Section#1{\Sec{\large #1} \setcounter{equation}{0} \vskip -6mm \indent}
\def\Sec{\@Startsection{section}{1}{\z@}
                                   {-3.5ex \@plus -1ex \@minus -.2ex}%
                                   {2.3ex \@plus.2ex}%
                                   {\normalfont\large\bfseries\boldmath}}
\def\@Startsection#1#2#3#4#5#6{%
  \if@noskipsec \leavevmode \fi
  \par
  \@tempskipa #4\relax
  \@afterindenttrue
  \ifdim \@tempskipa <\z@
    \@tempskipa -\@tempskipa \@afterindentfalse
  \fi
  \if@nobreak
    \everypar{}%
  \else
    \addpenalty\@secpenalty\addvspace\@tempskipa
  \fi
  \@ifstar
    {\@ssect{#3}{#4}{#5}{#6}}%
    {\@dblarg{\@Sect{#1}{#2}{#3}{#4}{#5}{#6}}}}
\def\@Sect#1#2#3#4#5#6[#7]#8{%
  \ifnum #2>\c@secnumdepth
    \let\@svsec\@empty
  \else
    \refstepcounter{#1}%
    \protected@edef\@svsec{\@seccntformat{#1}\relax}%
  \fi
  \@tempskipa #5\relax
  \ifdim \@tempskipa>\z@
    \begingroup
      #6{%
          \@hangfrom{\hskip #3\relax\@svsec \hskip -2.5mm}%
          \interlinepenalty \@M #8\@@par}
    \endgroup
    \csname #1mark\endcsname{#7}%
    \addcontentsline{toc}{#1}{%
      \ifnum #2>\c@secnumdepth \else
        \protect\numberline{\csname the#1\endcsname}%
      \fi
      #7}%
  \else
    \def\@svsechd{%
      #6{\hskip #3\relax
      \@svsec #8}%
      \csname #1mark\endcsname{#7}%
      \addcontentsline{toc}{#1}{%
        \ifnum #2>\c@secnumdepth \else
          \protect\numberline{\csname the#1\endcsname}%
        \fi
        #7}}%
  \fi
  \@xsect{#5}}
\renewenvironment{abstract}{%
        \small
        \quotation
         \noindent {\bfseries \abstractname } }%
      {\if@twocolumn\else\endquotation\fi}

\def\Subsec{\@StartSubsection{subsection}{2}{\z@}%
                                     {-3.25ex\@plus -1ex \@minus -.2ex}%
                                     {1.5ex \@plus .2ex}%
                                     {\normalfont\normalsize\bfseries\boldmath}}
\def\@StartSubsection#1#2#3#4#5#6{%
  \if@noskipsec \leavevmode \fi
  \par
  \@tempskipa #4\relax
  \@afterindenttrue
  \ifdim \@tempskipa <\z@
    \@tempskipa -\@tempskipa \@afterindentfalse
  \fi
  \if@nobreak
    \everypar{}%
  \else
    \addpenalty\@secpenalty\addvspace\@tempskipa
  \fi
  \@ifstar
    {\@ssect{#3}{#4}{#5}{#6}}%
    {\@dblarg{\@SubSect{#1}{#2}{#3}{#4}{#5}{#6}}}}
\def\@SubSect#1#2#3#4#5#6[#7]#8{%
  \ifnum #2>\c@secnumdepth
    \let\@svsec\@empty
  \else
    \refstepcounter{#1}%
    \protected@edef\@svsec{\@seccntformat{#1}\relax}%
  \fi
  \@tempskipa #5\relax
  \ifdim \@tempskipa>\z@
    \begingroup
      #6{%
          \@hangfrom{\hskip #3\relax\@svsec\hskip -1.5mm}%
          \interlinepenalty \@M #8\@@par}
    \endgroup
    \csname #1mark\endcsname{#7}%
    \addcontentsline{toc}{#1}{%
      \ifnum #2>\c@secnumdepth \else
        \protect\numberline{\csname the#1\endcsname}%
      \fi
      #7}%
  \else
    \def\@svsechd{%
      #6{\hskip #3\relax
      \@svsec #8}%
      \csname #1mark\endcsname{#7}%
      \addcontentsline{toc}{#1}{%
        \ifnum #2>\c@secnumdepth \else
          \protect\numberline{\csname the#1\endcsname}%
        \fi
        #7}}%
  \fi
  \@xsect{#5}}
\def\list#1#2{\ifnum \@listdepth >5\relax \@toodeep \else \global
\advance \@listdepth\@ne \fi \rightmargin \z@ \listparindent\z@
\itemindent\z@ \csname @list\romannumeral\the\@listdepth\endcsname
\def\@itemlabel{#1}\let\makelabel\@mklab \@nmbrlistfalse #2\relax
\@trivlist \parskip 0pt \parindent\listparindent \advance \linewidth
-\rightmargin \advance\linewidth -\leftmargin \advance\@totalleftmargin
\leftmargin \parshape \@ne \@totalleftmargin \linewidth \ignorespaces}
\renewcommand{\@makecaption}[2]{\begin{center}#1. #2\end{center}}
\catcode`@=12 
\theoremstyle{thmstyle}
\newtheorem{thm}{\indent Theorem}[section]
\newtheorem{lem}[thm]{\indent Lemma}
\newtheorem{prop}[thm]{\indent Proposition}

\newtheorem{prob}[thm]{\indent Problem}
\theoremstyle{remstyle}
\newtheorem{rem}[thm]{\indent Remark}
\newtheorem{algo}[thm]{\indent Algorithm}
\newtheorem{ex}[thm]{\indent Example}

\newsavebox{\mygraphic}
\def\pic{\begin{picture}(0,0) \put(-210,-1250){\usebox{\mygraphic}} \end{picture}}
\newfont{\HUGEbf}{cmbx10 scaled 3500}
\definecolor{gray}{rgb}{0.9,0.9,0.9}
\def\thebibliography#1{\section*{\bf \large References}
\list{[\arabic{enumi}]} {\settowidth \labelwidth{[#1]} \leftmargin
\labelwidth \advance \leftmargin \labelsep \usecounter{enumi}}
\def\newblock{\hskip .11em plus .33em minus .07em} \footnotesize \sloppy \clubpenalty
4000 \widowpenalty 4000 \sfcode`\.=1000 \relax}

\def\BR{\mathbb R}
\def\cA{\mathcal A}
\def\cB{\mathcal B}

\def\cH{\mathcal H}

\def\cL{\mathcal L}
\def\cP{\mathcal P}
\def\cU{\mathcal U}

\def\rd{\mathrm d}

\def\rdiv{\mathrm{div}}
\def\e{\mathrm e}
\def\err{\mathrm{err}}
\def\rand{\mathrm{rand}}

\def\s{\mathrm s}
\def\supp{\mathrm{supp}}
\def\true{\mathrm{true}}
\def\Ga{\Gamma}

\def\Om{\Omega}
\def\al{\alpha}
\def\be{\beta}
\def\ga{\gamma}
\def\de{\delta}
\def\ep{\epsilon}
\def\ve{\varepsilon}
\def\te{\theta}

\def\ka{\kappa}
\def\la{\lambda}
\def\si{\sigma}
\def\vp{\varphi}
\def\om{\omega}
\def\f{\frac}
\def\nb{\nabla}
\def\ov{\overline}
\def\pa{\partial}

\def\wt{\widetilde}
\def\tri{\triangle}
\def\beqnx{\begin{eqnarray*}} \def\eqnx{\end{eqnarray*}}

\theoremstyle{definition}

\numberwithin{equation}{section}

\title{\Large\bf\boldmath Theoretical Stability in Coefficient Inverse Problems\\
for General Hyperbolic Equations with\\
Numerical Reconstruction}

\author{\large Jie YU$^\dag$\qquad Yikan LIU$^{\ddag,*}$\qquad Masahiro YAMAMOTO$^\ddag$}

\date{}

\begin{document}

\maketitle

\thispagestyle{first}
\renewcommand{\thefootnote}{\fnsymbol{footnote}}

\footnotetext{\hspace*{-5mm} \begin{tabular}{@{}r@{}p{14cm}@{}} &
Manuscript last updated: \today.\\
$^\dag$ & School of Mathematical Sciences, Fudan University, No.\! 220 Handan Road, Shanghai 200433, China.\\
$^\ddag$ & Graduate School of Mathematical Sciences, The University of Tokyo, 3-8-1 Komaba, Meguro-ku, Tokyo 153-8914, Japan.\\
$^*$ & Corresponding author. E-mail: ykliu@ms.u-tokyo.ac.jp
\end{tabular}}

\renewcommand{\thefootnote}{\arabic{footnote}}

\begin{abstract}
In this article, we investigate the determination of the spatial component in the time-dependent second order coefficient of a hyperbolic equation from both theoretical and numerical aspects. By the Carleman estimates for general hyperbolic operators and an auxiliary Carleman estimate, we establish local H\"older stability with both partial boundary and interior measurements under certain geometrical conditions. For numerical reconstruction, we minimize a Tikhonov functional which penalizes the gradient of the unknown function. Based on the resulting variational equation, we design an iteration method which is updated by solving a Poisson equation at each step. One-dimensional prototype examples illustrate the numerical performance of the proposed iteration.

\vskip 4.5mm

\noindent\begin{tabular}{@{}l@{ }p{10cm}} {\bf Keywords } & Hyperbolic equation, Coefficient inverse problem, Carleman estimate, Local H\"older stability, Iteration method
\end{tabular}

\vskip 4.5mm


\end{abstract}

\baselineskip 14pt

\setlength{\parindent}{1.5em}

\setcounter{section}{0}

\Section{Introduction}\label{sec-intro}

Let $T>0$ and $\Om\subset\BR^n$ ($n=1,2,3$) be an open bounded domain whose boundary $\pa\Om$ is of $C^3$ class. Let $\nu=(\nu_1,\ldots,\nu_n)$ be the outward unit normal vector to $\pa\Om$, and denote the normal derivative on $\pa\Om$ by $\pa_\nu w:=\nb w\cdot\nu$. Introduce the hyperbolic operator
\begin{equation}\label{eq-def-Hp}
\cH_pw:=\pa_t^2w-\rdiv(p\,a\nb w)-b\cdot\nb w-c\,w\quad\mbox{in }Q:=\Om\times(-T,T),
\end{equation}
where $a=(a_{ij}(x,t))_{1\le i,j\le n}$ is a symmetric matrix, $b=(b_i(x,t))_{1\le i\le n}$ is a vector, and $p=p(x),c=c(x,t)$ are scalar functions. For the non-degeneracy, we assume that $p$ is strictly positive in $\ov\Om\,$, and $a$ is strictly positive-definite in $\ov Q\,$. For later use, we denote the normal derivative associated with the second order coefficient $p\,a$ as
\begin{equation}\label{eq-def-B}
\cB w:=p\,a\nb w\cdot\nu\quad\mbox{on }\pa\Om\times(-T,T).
\end{equation}

In this paper, we consider the initial value problem for a hyperbolic equation
\begin{equation}\label{eq-gov-u}
\begin{cases}
\cH_pu=F & \mbox{in }Q,\\
u=u_0,\ \pa_tu=u_1 & \mbox{in }\Om\times\{0\}.
\end{cases}
\end{equation}
From the physical point of view, \eqref{eq-gov-u} models the general acoustic wave in a highly anisotropic medium depending on both space and time. In a recent paper \cite{LY14}, it turns out that \eqref{eq-gov-u} also describes the one-dimensional time cone model for phase transformation. Note that for the well-posedness of the forward problem, we should attach \eqref{eq-gov-u} with a boundary condition. However, for the inverse problem proposed later, it suffices to access e.g.\! only the partial boundary value, which is regarded as a part of observation data. Therefore, for the moment we do not include any boundary condition in \eqref{eq-gov-u}. To emphasize the dependency, throughout this paper we will write any solution satisfying \eqref{eq-gov-u} with the coefficient $p$ as $u(p)$. Detailed assumptions on the coefficients $p,a,b,c,F,u_0,u_1$ involved in \eqref{eq-gov-u} will be given later in Section \ref{sec-main}.

This paper is mainly concerned with the following coefficient inverse problem on the determination of the spatial component $p$ in the principal part of the hyperbolic operator \eqref{eq-def-Hp}.

\begin{prob}\label{prob-cip}
Let $\Ga\subset\pa\Om$ be a subboundary, $\om\subset\Om$ be a subdomain, and $u(p)$ satisfy \eqref{eq-gov-u}. Determine the coefficient $p$ by

{\bf Type (I)}\ \ the partial boundary observation of $u(p)$ and $\cB u(p)$ on $\Ga\times(-T,T)$, or

{\bf Type (II)}\ \ the partial interior observation of $u(p)$ in $\om\times(-T,T)$.
\end{prob}

In the formulation of Problem \ref{prob-cip}, the second order coefficient $p(x)a(x,t)$ takes the form of incompletely separated variables, where the unknown spatial component $p(x)$ contributes an important part in the wave propagation speed. In view of the acoustic equation, Problem \ref{prob-cip} stands for the identification of the bulk modulus, which is of practical significance. Hence, we will not only investigate the theoretical aspect of Problem \ref{prob-cip} due to our interest in mathematics, but also consider the reconstruction method to solve $p(x)$ numerically.

In retrospect, researches on coefficient inverse problems for hyperbolic equations started soon after the pioneering work of Bukhgeim and Klibanov \cite{BK81} which discovered the potential of Carleman estimates. We refer e.g.\! to \cite{K87,K92,I93} for some early results mainly on the uniqueness. Around 2000s, \cite{Y99,IY01a,IY01b} established the global Lipschitz stability for determining the zeroth order coefficient $c(x)$ in $(\pa_t^2-\tri-c)u=0$ by the same types of data in Problem \ref{prob-cip}. For Problem \ref{prob-cip} with $a=I_{n\times n}$ and $b=c\equiv0$ in \eqref{eq-def-Hp}, Imanuvilov and Yamamoto \cite{IY03} employed an $H^{-1}$ Carleman estimate to obtain the global H\"older stability by partial boundary observation. Later, this result was improved to Lipschitz in Bellassoued and Yamamoto \cite{BY08} with time-independent coefficients, i.e., $a=(a_{ij}(x))$ in \eqref{eq-def-Hp}. For other references on this direction, see also \cite{BY06,I06}. It reveals that all of the above literature more or less imposes some geometrical conditions because of the finite wave propagation speed. Meanwhile, the above results mostly rely on a linearization approach, which reduces the problem to a corresponding inverse source problem. In addition, we emphasize the difference between e.g.\! $\rdiv(p\nb u)$ and $p\,\tri u$, where the former is more physical and the latter is technically easier (see \cite{B04}). However, there seems no publication treating Problem \ref{prob-cip} in the case of time-dependent principal parts due to the essential difficulty. Recently, Jiang, Liu and Yamamoto \cite{JLY17a} established the second order Carleman estimate for \eqref{eq-def-Hp} and proved the local H\"older stability for a related inverse source problem. Motivated by \cite{JLY17a}, we first attempt to generalize the result in \cite{IY03} with a more general hyperbolic operator $\cH_p$ with time-dependent coefficients, which is the first focus of this article.

Simultaneously, we have witnessed the recent applications of the iterative thresholding algorithm to inverse problems for partial differential equations. For the abstract formulation and convergence analysis of the algorithm, we refer to \cite{DDM04,DTV07,RT06}. Attracted by its efficiency and robustness in many image processing problems, \cite{JFZ14} first utilized the iterative thresholding algorithm to solve inverse problems for elliptic and parabolic equations. In \cite{LJY15,JLY17a,JLY17b}, similar iteration methods were implemented to treat inverse source problems for hyperbolic-type equations with different types of observation data. Following the same line, we also attempt to develop the same class of iteration method to solve Problem \ref{prob-cip} numerically, which is the second focus of this article. However, we should realize the underlying ill-posedness as well as the nonlinearity of Problem \ref{prob-cip}, which differs considerably from inverse source problems. For the numerical reconstruction of a time-dependent principal coefficient, we refer to \cite{LXY12}.

Based on the newly established Carleman estimate in \cite{JLY17a}, we first prove the local H\"older stability of Problem \ref{prob-cip} for both types of observation data (see Theorem \ref{thm-stab}). Due to the lack of an $H^{-1}$ Carleman estimate for $\cH_p$ as that in \cite{IY03}, we have to argue in an alternative way to evaluate the $H^1$-norm of the difference in unknown functions, which results from the divergence form $\rdiv(p\,a\nb u)$ in \eqref{eq-gov-u}. For the numerical reconstruction, we reformulate Problem \ref{prob-cip} as a minimization problem with the Tikhonov regularization penalizing the $L^2$-norm of $\nb p$. Deriving the variational equation of the minimizer, we arrive at a novel iteration method which needs to solve a Poisson equation at each step.

The rest of this article is organized as follows. Preparing the necessary ingredients including the key Carleman estimates, in Section \ref{sec-main} we state the main result on the theoretical stability of Problem \ref{prob-cip}. Then we give the proof of the main result in Section \ref{sec-proof}. Next, Section \ref{sec-recon} is devoted to the derivation of an iteration method for Problem \ref{prob-cip}, followed by Section \ref{sec-numer} illustrating several one-dimensional numerical examples. Finally, we provide some concluding remarks in Section \ref{sec-remark}.

\Section{Preliminaries and Main Results}\label{sec-main}

In this section, we start from the general settings and assumptions concerning the governing equation \eqref{eq-gov-u}, and prepare the key Carleman estimates for the hyperbolic operator \eqref{eq-def-Hp}. Then we state the main result of this paper, which gives the stability estimate for Problem \ref{prob-cip}.

Throughout this paper, we write $\pa_i=\f\pa{\pa x_i}$ and $\pa_i\pa_j=\f{\pa^2}{\pa x_i\pa x_j}$ ($1\le i,j\le n$) for the partial derivatives in space. We recall the definition $Q:=\Om\times(-T,T)$ and the notations $W^{k,\infty}(-T,T;W^{\ell,\infty}(\Om))$, $H^k(\Om)$, $H^{k-1/2}(\pa\Om)$, etc.\! ($k,\ell=0,1,\ldots$) for the usual Sobolev spaces (see Adams \cite{A75}). For the various coefficients appearing in \eqref{eq-gov-u}, we basically make the following assumptions:
\begin{equation}\label{eq-asp-0}
\begin{aligned}
& a\in(W^{3,\infty}(Q))^{n\times n},\quad\exists\,\ka_0>0\mbox{ such that }a\,\xi\cdot\xi\ge\ka_0|\xi|^2\mbox{ in }\ov Q\,,\ \forall\,\xi\in\BR^n,\\
& b\in (W^{2,\infty}(-T,T;L^\infty(\Om)))^n,\quad c\in W^{2,\infty}(-T,T;L^\infty(\Om)),\quad p\in W^{2,\infty}(\Om),\\
& F\in H^2(-T,T;L^2(\Om)),\quad u_0\in W^{3,\infty}(\Om),\quad u_1\in H^2(\Om).
\end{aligned}
\end{equation}
With some suitably given boundary condition and the compatibility condition, it is well known that the initial-boundary value problem governed by \eqref{eq-gov-u} admits a unique solution $u(p)$ which depends continuously on the involved coefficients (see e.g.\! \cite{LM72,JLY17a}). In order to prove the theoretical stability for Problem \ref{prob-cip}, we have to assume
\begin{equation}\label{eq-asp-reg}
u(p)\in\bigcap_{k=0}^2H^{4-k}(-T,T;H^k(\Om)),
\end{equation}
which satisfies the a priori estimate
\begin{equation}\label{eq-asp-bnd}
\sum_{k=0}^2\|u(p)\|_{H^{3-k}(-T,T;H^k(\Om))}\le M_0
\end{equation}
with a constant $M_0>0$. As was mentioned in \cite{JLY17a}, the smoothness of coefficients assumed in \eqref{eq-asp-0} may not guarantee the regularity in \eqref{eq-asp-reg}. Instead, we only understand \eqref{eq-asp-0}--\eqref{eq-asp-bnd} as the minimum necessary assumptions for the stability.

If the initial value problem \eqref{eq-gov-u} is only formulated in $\Om\times(0,T)$, we should additionally impose
\[\pa_ta=\pa_tb=\pa_tc=\pa_tF=u_1=0\quad\mbox{in }\Om\times\{0\}.\]
Then we can perform even reflections of $a,b,c,F$ and the solution $u(p)$ with respect to $t$, so that they preserve the same regularity in $Q$ as assumed in \eqref{eq-asp-0} and \eqref{eq-asp-reg}. To circumvent the above technical assumption, we simply consider \eqref{eq-gov-u} in $Q=\Om\times(-T,T)$ without loss of generality.

Next, we recall the Carleman estimates for the hyperbolic operator $\cH_p$ defined in \eqref{eq-def-Hp}, which play an essential role in both the statement and the proof of the main result. Similarly to \cite{JLY17a}, we pick $d\in C^2(\ov\Om)$ such that $d>0$ in $\ov\Om$\,, and set
\begin{equation}\label{eq-def-psi}
\psi(x,t):=d(x)-\be\,t^2,\quad0<\be<1.
\end{equation}
By $d(x)$ and $\psi(x,t)$, we introduce the level sets with a parameter $\de\ge0$ as
\begin{equation}\label{eq-def-level}
\Om(\de):=\{x\in\Om;\,d(x)>\de\},\quad Q(\de):=\{(x,t)\in Q;\,\psi(x,t)>\de\}.
\end{equation}
With a sufficiently large parameter $\la>0$, we define the weight function as
\begin{equation}\label{eq-def-weight}
\vp(x,t):=\e^{\la\psi(x,t)}.
\end{equation}
We collect the Carleman estimates concerning \eqref{eq-def-Hp} in the following lemma.

\begin{lem}[see \cite{JLY17a}]\label{lem-CE}
Let the coefficients $a,b,c$ and $p$ satisfy \eqref{eq-asp-0}. Suppose that the hyperbolic operator $\cH_p$ in \eqref{eq-def-Hp} admits the following Carleman estimate with the weight function $\vp$ in \eqref{eq-def-weight}: For any $\de\ge0$, there exists constants $s_0>0$ and $C_0>0$ such that
\begin{equation}\label{eq-CE-1}
\int_{Q(\de)}s\left(|\pa_tw|^2+|\nb w|^2+s^2|w|^2\right)\e^{2s\vp}\,\rd x\rd t\le C_0\int_{Q(\de)}|\cH_pw|^2\e^{2s\vp}\,\rd x\rd t
\end{equation}
holds for all $s\ge s_0$ and $w\in H^2(Q)$ satisfying $\supp\,w\subset\ov{Q(\de)}\,$, where $Q(\de)\ (\de\ge0)$ is defined in \eqref{eq-def-level}. Then for any $\de\ge0$, there exist constants $s_1>0$ and $C_1>0$ such that
\begin{align}
\int_{Q(\de)}\sum_{i,j=1}^n|\pa_i\pa_jw|^2\e^{2s\vp}\,\rd x\rd t & \le C_1\int_{Q(\de)}\left(\f1s|\pa_t(\cH_pw)|^2+s|\cH_pw|^2\right)\e^{2s\vp}\,\rd x\rd t,\label{eq-CE-2a}\\
\int_{Q(\de)}|\pa_t^2w|^2\e^{2s\vp}\,\rd x\rd t & \le C_1\int_{Q(\de)}\left(\f1s|\pa_t(\cH_pw)|^2+|\cH_pw|^2\right)\e^{2s\vp}\,\rd x\rd t\label{eq-CE-2b}
\end{align}
holds for all $s\ge s_1$ and $w\in H^2(Q)$ satisfying
\[\pa_tw\in H^2(Q),\quad\supp\,w\subset\ov{Q(\de)}\,,\quad\pa_t^kw(\,\cdot\,,\pm T)=0,\ k=0,1,2.\]
\end{lem}

\begin{rem}
In general, the first order Carleman estimate \eqref{eq-CE-1} does not hold automatically. Indeed, some artificial conditions on $d$ and the second order coefficient $p\,a$ are necessary to validate \eqref{eq-CE-1}. For example, in the isotropic case
\[p(x)a_{ij}(x,t)=\begin{cases}
a_0(x,t), & i=j,\\
0, & i\ne j,
\end{cases}\]
under certain geometrical assumptions and the choice
\[d(x)=x_1^2+\cdots+x_{n-1}^2+(x_n-1)^2,\]
a sufficient condition for \eqref{eq-CE-1} can be (see Isakov \cite[Theorem 3.4.3]{I06})
\[a_0>\sqrt\be\,,\quad\be\left(2-\f{t\,\pa_ta_0}{a_0}+\f{2|t\,\nb a_0|}{\sqrt{a_0}}\right)<2a_0-\nb a_0\cdot\nb d+\pa_na_0\quad\mbox{in }\ov Q\,,\]
where $\be\in(0,1)$ is the parameter in \eqref{eq-def-psi}. On the other hand, the second order Carleman estimates \eqref{eq-CE-2a}--\eqref{eq-CE-2b} was established in \cite[Lemma 3.1]{JLY17a} on the basis of \eqref{eq-CE-1}.
\end{rem}

Now we turn our attention to Problem \ref{prob-cip}. For the subboundary $\Ga\subset\pa\Om$, the subdomain $\om\subset\Om$ and $T>0$ where the observation is taken, we make the following assumption. Given a constant $\de_0\ge0$ and a function $d\in C^2(\ov\Om)$ such that $d>0$ in $\ov\Om\,$, we assume
\begin{equation}\label{eq-asp-obs}
\ov{\Om(\de_0)}\subset\Om\cup\Ga,\quad\pa\om\supset\Ga,\quad T^2>\|d\|_{C(\ov\Om)},
\end{equation}
where $\Om(\de_0)$ is the level set defined in \eqref{eq-def-level}. Although the above assumption looks technical, it generalizes the similar assumption in treating the same inverse problems for simpler wave equations (see \cite{IY01b,LJY15}), where the function $d$ was simply taken as $d(x)=|x-x_0|^2$ with $x_0\not\in\ov\Om\,$.

Finally, for the unknown function $p(x)$ to be determined, we define an admissible set $\cU$ in the following way. For given constants $M_1>0,\ka_1>0$ and given functions $h_0\in W^{2,\infty}(\Ga)$ and $h_1\in W^{1,\infty}(\Ga)$ where $\Ga$ satisfies \eqref{eq-asp-obs}, we restrict $p$ in
\begin{equation}\label{eq-def-adm}
\cU:=\{p\in W^{2,\infty}(\Om);\,\|p\|_{H^1(\Om)}\le M_1,\ p\ge\ka_1\mbox{ in }\ov\Om\,,\ p=h_0,\ \pa_\nu p=h_1\mbox{ on }\Ga\}.
\end{equation}

Collecting all necessary ingredients, now we are in a position to state the main theoretical result, which answers the stability issue of Problem \ref{prob-cip}.

\begin{thm}\label{thm-stab}
Let $\Ga\subset\pa\Om$, $\om\subset\Om$ and $T>0$ satisfy \eqref{eq-asp-obs} with arbitrarily fixed $\de_0\ge0$ and $d\in C^2(\ov\Om)$ such that $d>0$ in $\ov\Om\,$. Let $u(p)$ and $u(q)$ satisfy \eqref{eq-gov-u} with coefficients $p$ and $q$ respectively, where $a,b,c$ satisfy \eqref{eq-asp-0} and $p,q\in\cU$ defined in \eqref{eq-def-adm}. Suppose that $u(p),u(q)$ satisfy \eqref{eq-asp-reg}--\eqref{eq-asp-bnd}, and the hyperbolic operator $\cH_p$ admits the Carleman estimate \eqref{eq-CE-1} with the weight function \eqref{eq-def-weight}. Further assume that there exists a constant $\ka_2>0$ such that
\begin{equation}\label{eq-asp-1}
|a(\,\cdot\,,0)\nb u_0\cdot\nb d|\ge\ka_2\quad\mbox{a.e.\! in }\Om.
\end{equation}
Then for any $\de>\de_0$, there exist constants $C>0$ and $\te\in(0,1)$ depending on $a,b,c,d,u_0,\de,T,\Ga$ or $\om$ such that
\begin{equation}\label{eq-stab-a}
\|p-q\|_{H^1(\Om(\de))}\le CD+CM^{1-\te}D^\te,
\end{equation}
where $\Om(\de)$ is defined in \eqref{eq-def-level}, $M:=\max\{M_0,M_1\}$ with the a priori bounds $M_0,M_1$ in \eqref{eq-asp-bnd} and \eqref{eq-def-adm} respectively, and
\begin{equation}\label{eq-def-obs}
D:=\left\{\!\begin{alignedat}{2}
& \sum_{k=0}^2\|u(p)-u(q)\|_{H^{4-k}(-T,T;H^{k-1/2}(\Ga))}\\
& \quad+\|\cB u(p)-\cB u(q)\|_{H^2(-T,T;H^{1/2}(\Ga))}, & \quad & \mbox{Type }{\rm(I)},\\
& \sum_{k=0}^2\|u(p)-u(q)\|_{H^{4-k}(-T,T;H^k(\om))}, & \quad & \mbox{Type }{\rm(II)}.
\end{alignedat}\right.
\end{equation}
\end{thm}

\begin{rem}\label{rem-cip}
We discuss Theorem \ref{thm-stab} from various points of view.

(1)\ \ At a first glance, the statement of Theorem \ref{thm-stab} resembles that of \cite[Theorem 2.3]{JLY17a} to a large extent. Indeed, as the majority of existing literature did, here we also linearize the system to consider the governing equation of the difference $u(p)-u(q)$, which partially reduces Problem \ref{prob-cip} to an inverse source problem. Even though, Theorem \ref{thm-stab} is nontrivial because one should deal with the source term
\begin{equation}\label{eq-source}
\rdiv((p-q)a\nb u(q))
\end{equation}
as that in \cite{IY03}. However, in our case there seems no $H^{-1}$ Carleman estimate and thus we cannot evaluate \eqref{eq-source} simply by $\|p-q\|_{L^2(\Om)}$. As a result, we should treat \eqref{eq-source} in the $L^2$ sense, so that we have to argue more to dominate the $H^1$-norm of $p-q$. Hence, although \eqref{eq-stab-a} gives an $H^1$ estimate which looks stronger than that in \cite{IY03}, the fact is that we fail to provide an $L^2$ one instead.

(2)\ \ In Theorem \ref{thm-stab}, the unknown coefficient $p$ is restricted in the admissible set $\cU$. According to the definition \eqref{eq-def-adm}, it means that the partial Cauchy data of $p$ is known on the subboundary $\Ga$, which seems reasonable because the observation is taken there. Moreover, the boundary operator $\cB$ defined in \eqref{eq-def-B} becomes linear on $\Ga\times(-T,T)$ within the admissible set $\cU$, which allows us to write
\[\cB u(p)-\cB u(q)=h_0\,a\nb(u(p)-u(q))\cdot\nu=\cB(u(p)-u(q))\quad\mbox{on }\Ga\times(-T,T)\]
in the observation data \eqref{eq-def-obs}.

(3)\ \ We explain the geometry of the observation data in the assumption \eqref{eq-asp-obs}. Since $d$ is strictly positive in $\ov\Om\,$, by definition we know $\Om(\de)=\Om$ for $0\le\de<\min_{x\in\ov\Om}d(x)$. Therefore, if we fix $\de_0=0$, then \eqref{eq-asp-obs} requires $\Ga=\pa\Om$ and $\pa\om\supset\pa\Om$, that is, the observation should be taken on the whole boundary. Nevertheless, in this case \eqref{eq-stab-a} becomes a global H\"older estimate with sufficiently small $\de>0$. On the other hand, if one attempt to shrink the observation region, then the estimate \eqref{eq-stab-a} tends to be local. On the opposite extreme, since $\Om(\de)=\emptyset$ for any $\de\ge\|d\|_{C(\ov\Om)}$, our result becomes trivial with large $\de$. As a result, the choice of $d$ in the weight function plays a delicate role in the balance between the cost of measurements and the stability. The optimal choice of $d$ seems to be an interesting topic, but we will not discuss it in this paper.
\end{rem}

\Section{Proof of Theorem \ref{thm-stab}}\label{sec-proof}

In this section, we prove Theorem \ref{thm-stab} under the same assumptions therein. The proof basically follows the same line as that in \cite{JLY17a}, which relies heavily on the Carleman estimates in Lemma \ref{lem-CE}. However, many details are different from each other, and in our problem we should especially argue more for the $H^1$ estimate in \eqref{eq-stab-a}.

Henceforth, by $C>0$ we denote generic constants independent of the parameter $s>0$ in Carleman estimates and the observation data $D>0$, which may change line by line.

For clearness, we divide the proof into five steps.\medskip

{\bf Step 1}\ \ We start from the general preparation of notations and auxiliary functions. For simplicity, we abbreviate the hyperbolic operator $\cH_p$ as $\cH$ throughout this section. Fixing any $p,q\in\cU$, we set $f:=p-q$. Then it follows from the definition \eqref{eq-def-adm} that $f\in H^2(\Om)$ and $f=\pa_\nu f=0$ on $\Ga$. Together with the a priori bound, we conclude
\begin{equation}\label{eq-est-f0}
f\in H^2(\Om),\quad f=\nb f=0\mbox{ on }\Ga,\quad\|f\|_{H^1(\Om)}\le2M_1\le2M.
\end{equation}
Correspondingly, we set $v:=u(p)-u(q)$. According to \eqref{eq-gov-u}, \eqref{eq-asp-reg} and \eqref{eq-asp-bnd}, $v$ satisfies
\begin{align}
& \begin{cases}
\cH v=\cL f:=\rdiv(f\,a\nb u(q)) & \mbox{in }Q=\Om\times(-T,T),\\
v=\pa_tv=0 & \mbox{in }\Om\times\{0\},
\end{cases}\label{eq-gov-v}\\
& v\in\bigcap_{k=0}^2H^{4-k}(-T,T;H^k(\Om)),\quad\sum_{k=0}^2\|v\|_{H^{3-k}(-T,T;H^k(\Om))}\le2M_0\le2M.\label{eq-est-v}
\end{align}
We note that $\cL$ is a first order differential operator with respect to $f$.

By the Sobolev extension theorem, there exists $\wt v\in\bigcap_{k=0}^2H^{4-k}(-T,T;H^k(\Om))$ such that
\begin{align}
& \begin{cases}
\wt v=v,\ \cB\wt v=\cB v\quad\mbox{on }\Ga\times(-T,T), & \mbox{Type (I)},\\
\wt v=v\quad\mbox{in }\om\times(-T,T), & \mbox{Type (II)},
\end{cases}\label{eq-def-tv}\\
& \sum_{k=0}^2\|\wt v\|_{H^{3-k}(-T,T;H^k(\Om))}\le CM,\quad\sum_{k=0}^2\|\wt v\|_{H^{4-k}(-T,T;H^k(\Om))}\le CD.\label{eq-est-tv}
\end{align}
For later use, we further introduce
\[y_\ell:=\pa_t^\ell(v-\wt v\,),\quad\ell=0,1,2.\]
Then it is readily seen that $y_\ell\in\bigcap_{k=0}^2H^{4-\ell-k}(-T,T;H^k(\Om))$ ($\ell=0,1,2$), and $y_0,y_1,y_2$ satisfy
\begin{equation}\label{eq-gov-y}
\begin{aligned}
& \cH y_0=\cL f-\cH\wt v\,,\quad\cH y_1=\cA'y_0+\cL'f-\pa_t(\cH\wt v\,),\\
& \cH y_2=2\cA'y_1+\cA''y_0+\cL''f-\pa_t^2(\cH\wt v\,),
\end{aligned}
\end{equation}
where we define
\begin{alignat*}{2}
\cA'w & :=\rdiv(p(\pa_ta)\nb w)+(\pa_tb)\cdot\nb w+(\pa_tc)w, & \quad\cL'f & :=\pa_t\,\rdiv(f\,a\nb u(q)),\\
\cA''w & :=\rdiv(p(\pa_t^2a)\nb w)+(\pa_t^2b)\cdot\nb w+(\pa_t^2c)w, & \quad\cL''f & :=\pa_t^2\,\rdiv(f\,a\nb u(q)).
\end{alignat*}
Moreover, it follows from \eqref{eq-def-tv}, \eqref{eq-est-v} and \eqref{eq-est-tv} that for $\ell=0,1,2$,
\begin{align}
& \begin{cases}
y_\ell=\cB y_\ell=0\quad\mbox{on }\Ga\times(-T,T), & \mbox{Type (I)},\\
y_\ell=0\quad\mbox{in }\om\times(-T,T), & \mbox{Type (II)},
\end{cases}\label{eq-y-0}\\
& \sum_{k=0}^2\|y_\ell\|_{H^{3-\ell-k}(-T,T;H^k(\Om))}\le CM.\label{eq-est-y}
\end{align}

Note that $y_\ell$ ($\ell=0,1,2$) may not vanish outside the observation region, which prevents us from applying the Carleman estimates in Lemma \ref{lem-CE}. To this end, it is necessary to introduce a cutoff function as that in \cite{JLY17a}. Recall the constant $\de_0\ge0$ and the function $d\in C^2(\ov\Om)$ given in advance. For any $\de>\de_0$, we fix a constant $\de_1\in(\de_0,\de)$, e.g., $\de_1=\f{\de_0+\de}2$. With the level set $Q(\de)$ defined in \eqref{eq-def-level}, we define $\mu\in C^\infty(\ov Q)$ such that
\begin{align}
& 0\le\mu\le1,\quad\mu=\begin{cases}
1 & \mbox{in }Q(\de_1),\\
0 & \mbox{in }\ov Q\setminus Q(\de_0),
\end{cases}\label{eq-def-mu}\\
& \mu_0(x):=\mu(x,0)=\|\mu(x,\,\cdot\,)\|_{C[-T,T]},\ \forall\,x\in\ov\Om\,.\label{eq-def-mu0}
\end{align}
Especially, we see that the condition \eqref{eq-def-mu0} is possible by the definition \eqref{eq-def-psi} of $\psi$. Meanwhile, owing to the assumption $T^2>\|d\|_{C(\ov\Om)}$ in \eqref{eq-asp-obs}, we can choose $\be\in(0,1)$ such that $\be T^2>\|d\|_{C(\ov\Om)}$, indicating $\psi(\,\cdot\,,\pm T)<0$ in $\ov\Om\,$. Together with the assumption $\ov{\Om(\de_0)}\subset\Om\cup\Ga$ in \eqref{eq-asp-obs}, we conclude
\begin{equation}\label{eq-mu-0}
\supp\,\mu\subset\ov{Q(\de_0)}\subset(\Om\cup\Ga)\times(-T,T).
\end{equation}

Now we further set
\[z_\ell:=\mu\,y_\ell=\mu\,\pa_t^\ell(v-\wt v\,),\quad\ell=0,1,2.\]
Employing \eqref{eq-y-0}, \eqref{eq-mu-0} and the assumption $\pa\om\supset\Ga$ in \eqref{eq-asp-obs}, we obtain for $\ell=0,1,2$ that
\begin{equation}\label{eq-cond-z}
\supp\,z_\ell\subset\ov{Q(\de_0)}\,,\quad\pa_t^kz_\ell(\,\cdot\,,\pm T)=0\ (k=0,1,2),\quad z_\ell=\cB z_\ell=0\mbox{ on }\pa\Om\times(-T,T).
\end{equation}
On the other hand, it is obvious that $z_\ell$ share the same regularity as that of $y_\ell$, i.e., $z_\ell\in H^2(Q)$ ($\ell=0,1,2$) and especially $\pa_tz_0,\pa_tz_1\in H^2(Q)$. By \eqref{eq-gov-y} and direct calculations, we deduce
\begin{align}
\cH z_0 & =\mu\,(\cL f-\cH\wt v\,)+[\cH,\mu]y_0=:F_0,\label{eq-gov-z0}\\
\cH z_1 & =\mu\,(\cL'f-\pa_t(\cH\wt v\,))+\cA'z_0+[\cH,\mu]y_1-[\cA',\mu]y_0=:F_1,\label{eq-gov-z1}\\
\cH z_2 & =\mu\,(\cL''f-\pa_t^2(\cH\wt v\,))+2\cA'z_1+\cA''z_0+[\cH,\mu]y_2-2[\cA',\mu]y_1-[\cA'',\mu]y_0=:F_2,\label{eq-gov-z2}
\end{align}
where
\begin{align*}
[\cH,\mu]w & :=2\,\{(\pa_t\mu)\pa_tw-p\,a\nb\mu\cdot\nb w\}+\left\{\pa_t^2\mu-\rdiv(p\,a\nb\mu)-b\cdot\nb\mu\right\}w,\\
[\cA',\mu]w & :=2\,p(\pa_ta)\nb\mu\cdot\nb w+\left\{\rdiv(p(\pa_ta)\nb\mu)+(\pa_tb)\cdot\nb\mu\right\}w,\\
[\cA'',\mu]w & :=2\,p(\pa_t^2a)\nb\mu\cdot\nb w+\left\{\rdiv(p(\pa_t^2a)\nb\mu)+(\pa_t^2b)\cdot\nb\mu\right\}w.
\end{align*}
Similarly to \cite{JLY17a}, it turns out that $[\cH,\mu]$, $[\cA',\mu]$ and $[\cA'',\mu]$ are all first order differential operators which only involve derivatives of $\mu$. Then the definition \eqref{eq-def-mu} of $\mu$ implies
\begin{equation}\label{eq-cmt-0}
[\cH,\mu]w=[\cA',\mu]w=[\cA'',\mu]w=0\quad\mbox{in }\ov{Q(\de_1)}\cup\left(\ov Q\setminus Q(\de_0)\right).
\end{equation}

{\bf Step 2}\ \ Now that $z_\ell$ ($\ell=0,1,2$) satisfy \eqref{eq-cond-z}, we are well prepared to apply the Carleman estimates in Lemma \ref{lem-CE} to $z_\ell$. We will utilize estimates \eqref{eq-est-f0}, \eqref{eq-est-tv} and \eqref{eq-est-y} repeatedly throughout the proof, and take advantage of the properties of the cutoff function $\mu$ as well as the weight function $\vp$.

Applying the first order Carleman estimate \eqref{eq-CE-1} in Lemma \ref{lem-CE} to \eqref{eq-gov-z2}, we have
\begin{equation}\label{eq-est-z2}
\int_{Q(\de_0)}s\left(|\pa_tz_2|^2+|\nb z_2|^2+s^2|z_2|^2\right)\e^{2s\vp}\,\rd x\rd t\le C\|F_2\,\e^{s\vp}\|_{L^2(Q(\de_0))}^2,\quad\forall\,s\gg1.
\end{equation}
Our aim in this step is to give an estimate for $\|F_2\,\e^{s\vp}\|_{L^2(Q(\de_0))}$. To this end, we apply \eqref{eq-CE-1}--\eqref{eq-CE-2a} in Lemma \ref{lem-CE} to \eqref{eq-gov-z0}--\eqref{eq-gov-z1} to obtain for $\ell=0,1$ that
\begin{align}
\sum_{|\ga|\le2}\|(\pa_x^\ga z_\ell)\e^{s\vp}\|_{L^2(Q(\de_0))}^2 & =\int_{Q(\de_0)}\left(\sum_{i,j=1}^n|\pa_i\pa_jz_\ell|^2+|\nb z_\ell|^2+|z_\ell|^2\right)\e^{2s\vp}\,\rd x\rd t\nonumber\\
& \le C\int_{Q(\de_0)}\left(\f1s|\pa_tF_\ell|^2+s|F_\ell|^2\right)\e^{2s\vp}\,\rd x\rd t,\quad\forall\,s\gg1.\label{eq-est-z01}
\end{align}
To proceed, we should estimate $\|(\pa_t^kF_\ell)\e^{s\vp}\|_{L^2(Q(\de_0))}$ for $k,\ell=0,1$. First, we combine the properties \eqref{eq-def-mu}--\eqref{eq-def-mu0} with \eqref{eq-cmt-0}, \eqref{eq-est-tv} and \eqref{eq-est-y} to dominate
\begin{align}
\|F_0\,\e^{s\vp}\|_{L^2(Q(\de_0))}^2 & \le C\int_{Q(\de_0)}\left(|\mu\,\cL f|^2+|\mu\,\cH\wt v|^2+|[\cH,\mu]y_0|^2\right)\e^{2s\vp}\,\rd x\rd t\nonumber\\
& \le C\int_{Q(\de_0)}\mu^2\left(|f|^2+|\nb f|^2\right)\e^{2s\vp}\,\rd x\rd t+C\exp\left(2s\max_{\ov{Q(\de_0)}}\vp\right)\|\cH\wt v\|_{L^2(Q(\de_0))}^2\nonumber\\
& \quad\,+C\int_{Q(\de_0)\setminus Q(\de_1)}|[\cH,\mu]y_0|^2\e^{2s\vp}\,\rd x\rd t\nonumber\\
& \le C\int_Q\mu^2\left(|f|^2+|\nb f|^2\right)\e^{2s\vp}\,\rd x\rd t+C\,\e^{Cs}\|\wt v\|_{H^2(Q)}^2\nonumber\\
& \quad\,+C\exp\left(2s\max_{\ov{Q(\de_0)}\setminus Q(\de_1)}\vp\right)\|y_0\|_{H^1(Q)}^2\nonumber\\
& \le C\int_Q\mu_0^2\left(|f|^2+|\nb f|^2\right)\e^{2s\vp}\,\rd x\rd t+C\,\e^{Cs}D^2+C\,\e^{2\eta_1s}M^2,\label{eq-est-F0}
\end{align}
where we used $\vp=\e^{\la\psi}\le\e^{\la\de_1}=:\eta_1$ in $\ov{Q(\de_0)}\setminus Q(\de_1)$ by the definition \eqref{eq-def-level}. Similarly, by
\[\pa_tF_0=\mu\,(\cL'f-\pa_t(\cH\wt v\,))+(\pa_t\mu)(\cL f-\cH\wt v\,)+\pa_t([\cH,\mu]y_0)\]
we further estimate
\begin{align}
\|(\pa_tF_0)\e^{s\vp}\|_{L^2(Q(\de_0))}^2 & \le C\int_{Q(\de_0)}\mu^2\left(|f|^2+|\nb f|^2\right)\e^{2s\vp}\,\rd x\rd t+C\,\e^{Cs}\|\pa_t(\cH\wt v\,)\|_{L^2(Q)}^2\nonumber\\
& \quad\,+C\,\e^{2\eta_1s}\left(\|f\|_{H^1(\Om)}^2+\|\cH\wt v\|_{L^2(Q)}^2+\|y_0\|_{H^2(Q)}^2\right)\nonumber\\
& \le C\int_Q\mu_0^2\left(|f|^2+|\nb f|^2\right)\e^{2s\vp}\,\rd x\rd t+C\,\e^{Cs}D^2+C\,\e^{2\eta_1s}M^2,\label{eq-est-tF0}
\end{align}
where we turned to the a priori estimate \eqref{eq-est-f0} to treat the term $(\pa_t\mu)\cL f$. Hence, substituting \eqref{eq-est-F0}--\eqref{eq-est-tF0} into \eqref{eq-est-z01} with $\ell=0$ yields
\begin{align}
\sum_{|\ga|\le2}\|(\pa_x^\ga z_0)\e^{s\vp}\|_{L^2(Q(\de_0))}^2 & \le Cs\int_Q\mu_0^2\left(|f|^2+|\nb f|^2\right)\e^{2s\vp}\,\rd x\rd t\nonumber\\
& \quad\,+C\,\e^{Cs}D^2+Cs\,\e^{2\eta_1s}M^2,\quad\forall\,s\gg1.\label{eq-est-z0}
\end{align}

Next, we deal with $F_1$ defined in \eqref{eq-gov-z1}. Noting the fact that $\cA'$ is a second order differential operator in space, we apply \eqref{eq-est-z0}, \eqref{eq-est-tv} and \eqref{eq-est-y} to derive
\begin{align}
\|F_1\,\e^{s\vp}\|_{L^2(Q(\de_0))}^2 & \le C\int_{Q(\de_0)}\mu^2\left(|f|^2+|\nb f|^2\right)\e^{2s\vp}\,\rd x\rd t+C\sum_{|\ga|\le2}\|(\pa_x^\ga z_0)\e^{s\vp}\|_{L^2(Q(\de_0))}^2\nonumber\\
& \quad+C\,\e^{Cs}\|\pa_t(\cH\wt v\,)\|_{L^2(Q)}^2+C\,\e^{2\eta_1s}\left(\|y_1\|_{H^1(Q)}+\|y_0\|_{H^1(Q)}\right)\nonumber\\
& \le Cs\int_Q\mu_0^2\left(|f|^2+|\nb f|^2\right)\e^{2s\vp}\,\rd x\rd t+C\,\e^{Cs}D^2+C\,\e^{2\eta_1s}M^2\label{eq-est-F1}
\end{align}
for all $s\gg1$. To bound $\|(\pa_tF_1)\e^{s\vp}\|_{L^2(Q(\de_0))}$, we calculate
\begin{align*}
\pa_tF_1 & =\cA'z_1+\cA''z_0+\mu\,(\cL''f-\pa_t^2(\cH\wt v\,))+(\pa_t\mu)(\cL'f-\pa_t(\cH\wt v\,))\\
& \quad\,+\pa_t([\cH,\mu]y_1-[\cA',\mu]y_0)+\cA'((\pa_t\mu)y_0).
\end{align*}
Similarly to the treatment for $\pa_tF_0$, we employ \eqref{eq-est-z0} again to estimate
\begin{align}
& \quad\,\,\|(\pa_tF_1)\e^{s\vp}\|_{L^2(Q(\de_0))}^2\nonumber\\
& \le C\sum_{\ell=0}^1\sum_{|\ga|\le2}\|(\pa_t^\ga z_\ell)\e^{s\vp}\|_{L^2(Q(\de_0))}^2+C\int_Q\mu_0^2\left(|f|^2+|\nb f|^2\right)\e^{2s\vp}\,\rd x\rd t+C\,\e^{Cs}\|\pa_t^2(\cH\wt v\,)\|_{L^2(Q)}^2\nonumber\\
& \quad\,+C\,\e^{2\eta_1s}\left(\|f\|_{H^1(\Om)}+\|\pa_t(\cH\wt v\,)\|_{L^2(Q)}^2+\|y_1\|_{H^2(Q)}^2+\|y_0\|_{H^2(Q)}^2\right)\nonumber\\
& \le C\sum_{|\ga|\le2}\|(\pa_x^\ga z_1)\e^{s\vp}\|_{L^2(Q(\de_0))}^2+Cs\int_Q\mu_0^2\left(|f|^2+|\nb f|^2\right)\e^{2s\vp}\,\rd x\rd t\nonumber\\
& \quad\,+C\,\e^{Cs}D^2+C\,\e^{2\eta_1s}M^2,\quad\forall\,s\gg1.\label{eq-est-tF1}
\end{align}
Substituting \eqref{eq-est-F1}--\eqref{eq-est-tF1} into \eqref{eq-est-z01} with $\ell=1$, we deduce
\begin{align*}
\sum_{|\ga|\le2}\|(\pa_x^\ga z_1)\e^{s\vp}\|_{L^2(Q(\de_0))}^2 & \le\f Cs\sum_{|\ga|\le2}\|(\pa_x^\ga z_1)\e^{s\vp}\|_{L^2(Q(\de_0))}^2+Cs^2\int_Q\mu_0^2\left(|f|^2+|\nb f|^2\right)\e^{2s\vp}\,\rd x\rd t\\
& \quad\,+C\,\e^{Cs}D^2+Cs^2\e^{2\eta_1s}M^2,\quad\forall\,s\gg1.
\end{align*}
Choosing $s>0$ sufficiently large, we can absorb the first term on the right-hand side into the left-hand side and conclude
\begin{align}
\sum_{|\ga|\le2}\|(\pa_x^\ga z_1)\e^{s\vp}\|_{L^2(Q(\de_0))}^2 & \le Cs^2\int_Q\mu_0^2\left(|f|^2+|\nb f|^2\right)\e^{2s\vp}\,\rd x\rd t\nonumber\\
& \quad\,+C\,\e^{Cs}D^2+Cs^2\e^{2\eta_1s}M^2,\quad\forall\,s\gg1.\label{eq-est-z1}
\end{align}
Using \eqref{eq-est-tv}, \eqref{eq-est-y} and \eqref{eq-cmt-0} again, we apply \eqref{eq-est-z0} and \eqref{eq-est-z1} to the definition \eqref{eq-gov-z2} of $F_2$ and arrive at the estimate
\begin{align}
\|F_2\,\e^{s\vp}\|_{L^2(Q(\de_0))}^2 & \le C\int_Q\mu_0^2\left(|f|^2+|\nb f|^2\right)\e^{2s\vp}\,\rd x\rd t+C\,\e^{Cs}\|\pa_t^2(\cH\wt v\,)\|_{L^2(Q)}^2\nonumber\\
& \quad\,+C\sum_{\ell=0}^1\sum_{|\ga|\le2}\|(\pa_x^\ga z_\ell)\e^{s\vp}\|_{L^2(Q(\de_0))}^2+C\,\e^{2\eta_1s}\sum_{\ell=0}^2\|y_\ell\|_{H^1(Q)}^2\nonumber\\
& \le Cs^2\int_Q\mu_0^2\left(|f|^2+|\nb f|^2\right)\e^{2s\vp}\,\rd x\rd t+C\,\e^{Cs}D^2+Cs^2\e^{2\eta_1s}M^2\label{eq-est-F2}
\end{align}
for all $s\gg1$.\medskip

{\bf Step 3}\ \ We further introduce
\[\cP w:=\cH w+b\cdot\nb w+c\,w=\pa_t^2w-\rdiv(p\,a\nb w),\quad \wt z:=\e^{s\vp}z_2.\]
Then simple calculations yield
\begin{align}
& \pa_t\wt z=\e^{s\vp}(\pa_tz_2+s(\pa_t\vp)z_2),\quad\nb\wt z=\e^{s\vp}(\nb z_2+s\,z_2\nb\vp),\label{eq-z-z2}\\
& \cP\wt z=\e^{s\vp}\left\{\cP z_2+2s\,((\pa_t\vp)\pa_tz_2-p\,a\nb\vp\cdot\nb z_2)+2\left(\cP\vp+s(|\pa_t\vp|^2-p\,a\nb\vp\cdot\nb\vp)\right)z_2\right\}.\label{eq-z-z2'}
\end{align}

We attempt to give upper and lower estimates for
\[I_0:=2\int_{-T}^0\!\int_\Om(\pa_t\wt z\,)(\cP\wt z\,)\,\rd x\rd t.\]
First, since $\cP z_2=F_2+b\cdot\nb z_2+c\,z_2$, the application of \eqref{eq-est-z2} immediately gives
\[\int_{Q(\de_0)}|\cP z_2|^2\,\e^{2s\vp}\,\rd x\rd t\le\int_{Q(\de_0)}\left\{|F_2|^2+C(|\nb z_2|^2+|z_2|^2)\right\}\e^{2s\vp}\,\rd x\rd t\le C\|F_2\,\e^{s\vp}\|_{L^2(Q(\de_0))}^2\]
for all $s\gg1$. Noticing $\supp\,z=\supp\,z_2\subset\ov{Q(\de_0)}$ and using \eqref{eq-est-z2} again, we employ \eqref{eq-z-z2}--\eqref{eq-z-z2'} to estimate $I_0$ from above as
\begin{align}
I_0 & \le2\int_{-T}^0\!\int_\Om|\pa_t\wt z\,||\cP\wt z\,|\,\rd x\rd t\le2\int_{Q(\de_0)}|\pa_t\wt z\,||\cP\wt z\,|\,\rd x\rd t\nonumber\\
& \le C\int_{Q(\de_0)}(|\pa_tz_2|+s|z_2|)\,\{|\cP z_2|+Cs\,(|\pa_tz_2|+|\nb z_2|+s|z_2|)\}\,\e^{2s\vp}\,\rd x\rd t\nonumber\\
& \le C\int_{Q(\de_0)}\left\{|\cP z_2|^2+s\left(|\pa_tz_2|^2+|\nb z_2|^2\right)+s^3|z_2|^2\right\}\e^{2s\vp}\,\rd x\rd t\nonumber\\
& \le C\|F_2\,\e^{s\vp}\|_{L^2(Q(\de_0))}^2,\quad\forall\,s\gg1.\label{eq-est-I0u}
\end{align}
On the other hand, it follows from \eqref{eq-cond-z} that
\[\wt z=0\mbox{ on }\pa\Om\times(-T,T),\quad\wt z=\pa_t\wt z=0\mbox{ in }\Om\times\{-T\},\]
which allows us to perform integration by parts to estimate $I_0$ from below as
\begin{align}
I_0 & =\int_\Om\int_{-T}^0\pa_t\left(|\pa_t\wt z\,|^2\right)\rd t\rd x-2\int_{-T}^0\!\int_{\pa\Om}(\pa_t\wt z\,)(\cB\wt z\,)\,\rd\sigma\rd t+2\int_{-T}^0\!\int_\Om p\,a\nb\wt z\cdot\nb(\pa_t\wt z\,)\,\rd x\rd t\nonumber\\
& =\|\pa_t\wt z(\,\cdot\,,0)\|_{L^2(\Om)}^2+\sum_{i,j=1}^n\int_\Om p\int_{-T}^0a_{ij}\,\pa_t((\pa_i\wt z\,)\pa_j\wt z\,)\rd t\rd x\nonumber\\
& \ge\sum_{i,j=1}^n\int_\Om p\left\{(a_{ij}(\pa_i\wt z\,)\pa_j\wt z\,)(\,\cdot\,,0)-\int_{-T}^0(\pa_ta_{ij})(\pa_i\wt z\,)\pa_j\wt z\,\rd t\right\}\rd x\nonumber\\
& =\int_\Om p\,(a\nb\wt z\cdot\nb\wt z\,)(\,\cdot\,,0)\,\rd x-\int_{-T}^0\!\int_\Om p(\pa_ta)\nb\wt z\cdot\nb\wt z\,\rd x\rd t\nonumber\\
& \ge\ka_0\ka_1\|\nb\wt z(\,\cdot\,,0)\|_{L^2(\Om)}^2-C\int_{Q(\de_0)}|\nb\wt z\,|^2\,\rd x\rd t,\label{eq-est-I0d}
\end{align}
where we applied the lower bounds in \eqref{eq-asp-0} and \eqref{eq-def-adm} to obtain \eqref{eq-est-I0d}. Utilizing \eqref{eq-z-z2} and \eqref{eq-est-z2} again, we further estimate
\begin{align*}
\int_{Q(\de_0)}|\nb\wt z\,|^2\,\rd x\rd t & \le C\int_{Q(\de_0)}\left(|\nb z_2|+Cs|z_2|\right)^2\e^{2s\vp}\,\rd x\rd t\le C\int_{Q(\de_0)}\left(|\nb z_2|^2+s^2|z_2|^2\right)\e^{2s\vp}\,\rd x\rd t\\
& \le\f Cs\|F_2\,\e^{s\vp}\|_{L^2(Q(\de_0))}^2,\quad\forall\,s\gg1.
\end{align*}
Combining the above inequality with \eqref{eq-est-I0d}, \eqref{eq-est-I0u} and \eqref{eq-est-F2}, we obtain for all $s\gg1$ that
\begin{align}
\|\nb\wt z(\,\cdot\,,0)\|_{L^2(\Om)}^2 & \le C\|F_2\,\e^{s\vp}\|_{L^2(Q(\de_0))}^2+\f Cs\|F_2\,\e^{s\vp}\|_{L^2(Q(\de_0))}^2\le C\|F_2\,\e^{s\vp}\|_{L^2(Q(\de_0))}^2\nonumber\\
& \le Cs^2\int_{Q(\de_0)}\mu_0^2\left(|f|^2+|\nb f|^2\right)\e^{2s\vp}\,\rd x\rd t+C\,\e^{Cs}D^2+Cs^2\,\e^{2\eta_1s}M^2.\label{eq-est-wtz}
\end{align}

Next, we shall relate the above estimate of $\nb\wt z(\,\cdot\,,0)$ with that of $f$. In fact, by the definition of $\wt z$, we take $t=0$ in \eqref{eq-gov-v} and find
\[\wt z(\,\cdot\,,0)=(\e^{s\vp}z_2)(\,\cdot\,,0)=\e^{s\vp_0}\left(\mu\,\pa_t^2(v-\wt v\,)\right)(\,\cdot\,,0)=\e^{s\vp_0}\mu_0\left(\cL_0f-\pa_t^2\wt v(\,\cdot\,,0)\right),\]
where
\begin{align}
& \vp_0(x):=\vp(x,0)=\e^{\la d(x)},\label{eq-def-phi0}\\
& \cL_0f:=\rdiv(f\,a(\,\cdot\,,0)\nb u_0)=a(\,\cdot\,,0)\nb u_0\cdot\nb f+\rdiv(a(\,\cdot\,,0)\nb u_0)\,f.\label{eq-def-L0}
\end{align}
Hence, by $\cL_0(\mu_0f)=\mu_0\,\cL_0f+f\,a(\,\cdot\,,0)\nb u_0\cdot\nb\mu_0$, we obtain
\begin{align*}
\nb(\e^{s\vp_0}\cL_0(\mu_0f)) & =\nb\wt z(\,\cdot\,,0)+\nb\left\{\e^{s\vp_0}\left(\mu_0\,\pa_t^2\wt v(\,\cdot\,,0)+f\,a(\,\cdot\,,0)\nb u_0\cdot\nb\vp_0\right)\right\}\\
& =\nb\wt z(\,\cdot\,,0)+\e^{s\vp_0}\big[\mu_0\left\{s\left(\pa_t^2\wt v(\,\cdot\,,0)\right)\nb\vp_0+\nb\pa_t^2\wt v(\,\cdot\,,0)\right\}+\left(\pa_t^2\wt v(\,\cdot\,,0)\right)\nb\mu_0\\
& \quad\,+s\,(f\,a(\,\cdot\,,0)\nb u_0\cdot\nb\mu_0)\nb\vp_0+\nb(f\,a(\,\cdot\,,0)\nb u_0\cdot\nb\mu_0)\big].
\end{align*}
Recalling the definition \eqref{eq-def-level} of $\Om(\de)$, we see $\supp(\nb\mu_0)\subset\ov{\Om(\de_0)}\setminus\Om(\de_1)$ and $\vp_0\le\eta_1$ in $\ov{\Om(\de_0)}\setminus\Om(\de_1)$. By the same argument as before, we apply \eqref{eq-est-wtz} to bound
\begin{align}
\int_\Om\left|\nb(\e^{s\vp_0}\cL_0(\mu_0f))\right|^2\,\rd x & \le2\|\nb\wt z(\,\cdot\,,0)\|_{L^2(\Om)}^2+C\,\e^{Cs}\int_\Om\left(s^2|\pa_t^2\wt v(\,\cdot\,,0)|^2+|\nb\pa_t^2\wt v(\,\cdot\,,0)|^2\right)\rd x\nonumber\\
& \quad\,+C\int_{\Om(\de_0)\setminus\Om(\de_1)}\left\{|\pa_t^2\wt v(\,\cdot\,,0)|^2+s^2|f|^2+\left(|f|^2+|\nb f|^2\right)\right\}\e^{2s\vp_0}\,\rd x\nonumber\\
& \le2\|\nb\wt z(\,\cdot\,,0)\|_{L^2(\Om)}^2+C\,\e^{Cs}D^2+Cs^2\,\e^{2\eta_1s}M^2\nonumber\\
& \le Cs^2\int_{Q(\de_0)}\mu_0^2\left(|f|^2+|\nb f|^2\right)\e^{2s\vp}\,\rd x\rd t\nonumber\\
& \quad\,+C\,\e^{Cs}D^2+Cs^2\,\e^{2\eta_1s}M^2,\quad\forall\,s\gg1,\label{eq-est-L0f}
\end{align}
where we used the Sobolev embedding $C[-T,T]\subset H^1(-T,T)$ and \eqref{eq-est-tv} to dominate
\[\|\pa_t^2\wt v(\,\cdot\,,0)\|_{H^\ell(\Om)}\le\|\pa_t^2\wt v\|_{C([-T,T];H^\ell(\Om))}\le C\|\pa_t^2\wt v\|_{H^1(-T,T;H^\ell(\Om))}\le\begin{cases}
CM, & \ell=0,\\
CD, & \ell=1.
\end{cases}\]

{\bf Step 4}\ \ In order to relate the above estimate \eqref{eq-est-L0f} with the $H^1$-norm of $f$, we need the following Carleman estimate for the first order differential operator $\cL_0$ in \eqref{eq-def-L0}.

\begin{lem}\label{lem-CE-L0}
Let $\cL_0$ and $\vp_0$ be defined in \eqref{eq-def-L0} and \eqref{eq-def-phi0} respectively, and assume \eqref{eq-asp-1}. Then there exist constants $C>0$ and $s_2>0$ such that
\[\int_\Om\left(|g|^2+|\nb g|^2\right)\e^{2s\vp_0}\,\rd x\le C\int_\Om|\nb(\e^{s\vp_0}\cL_0g)|^2\,\rd x\]
holds for all $s\ge s_2$ and $g\in H_0^2(\Om)$.
\end{lem}	

\begin{proof}
First we show that there exist constants $C>0$ and $s_3>0$ such that
\begin{equation}\label{eq-est-L0-0}
s^2\int_\Om|g|^2\e^{2s\vp_0}\,\rd x\le C\int_\Om|\cL_0g|^2\e^{2s\vp_0}\,\rd x,\quad\forall\,s\ge s_3,\ \forall\,g\in H_0^1(\Om)
\end{equation}
In fact, by setting $\wt g:=\e^{s\vp_0}g$, we calculate
\[e^{s\vp_0}\cL_0g=\e^{s\vp_0}\cL_0(\e^{-s\vp_0}\wt g\,)=\cL_0\wt g-s\,(a(\,\cdot\,,0)\nb u_0\cdot\nb\vp_0)\,\wt g\,.\]
By the assumption \eqref{eq-asp-1} and $d>0$ in $\ov\Om\,$, we obtain the lower bound
\[|a(\,\cdot\,,0)\nb u_0\cdot\nb\vp_0|=\la\,\e^{\la d}|a(\,\cdot\,,0)\nb u_0\cdot\nb d|\ge\la\ka_2.\]
Then we can estimate
\begin{align*}
\int_\Om|\cL_0g|^2\e^{2s\vp_0}\,\rd x & =\int_\Om|\cL_0\wt g-s\,(a(\,\cdot\,,0)\nb u_0\cdot\nb\vp_0)\,\wt g|^2\,\rd x\\
& \ge s^2\int_\Om|a(\,\cdot\,,0)\nb u_0\cdot\nb\vp_0|^2|\wt g|^2\,\rd x-s\,I_1\ge(\la\ka_2)^2s^2\int_\Om|\wt g|^2\,\rd x-s\,I_1,
\end{align*}
where
\[I_1:=-2\int_\Om(a(\,\cdot\,,0)\nb u_0\cdot\nb\vp_0)(\cL_0\wt g\,)\,\wt g\,\rd x.\]
By $g\in H_0^1(\Om)$, we perform integration by parts to treat $I_1$ as
\begin{align*}
I_1 & =-2\int_\Om(a(\,\cdot\,,0)\nb u_0\cdot\nb\vp_0)\left\{(a(\,\cdot\,,0)\nb u_0\cdot\nb\wt g\,)\,\wt g+\rdiv(a(\,\cdot\,,0)\nb u_0)|\wt g|^2\right\}\rd x\\
& =\int_\Om\{\rdiv((a(\,\cdot\,,0)\nb u_0\cdot\nb\vp_0)a(\,\cdot\,,0)\nb u_0)-2(a(\,\cdot\,,0)\nb u_0\cdot\nb\vp_0)\,\rdiv(a(\,\cdot\,,0)\nb u_0)\}|\wt g|^2\,\rd x
\end{align*}
According to the assumption \eqref{eq-asp-0}, both $\|a(\,\cdot\,,0)\|_{W^{1,\infty}(\Om)}$ and $\|u_0\|_{W^{2,\infty}(\Om)}$ are bounded, which indicates $|I_1|\le C\|\wt g\|_{L^2(\Om)}^2$ and thus
\[\int_\Om|\cL_0g|^2\e^{2s\vp_0}\,\rd x\ge\left((\la\ka_2)^2s^2-Cs\right)\int_\Om|\wt g|^2\,\rd x.\]
Therefore, there exists a constant $s_3>0$ such that the right-hand side is strictly positive for all $s\ge s_3$, which implies \eqref{eq-est-L0-0}.

Next, since $g\in H_0^2(\Om)$ gives $\nb g\in(H_0^1(\Om))^n$, we apply \eqref{eq-est-L0-0} to $\nb g$ to obtain
\begin{equation}\label{eq-est-L0-1}
s^2\int_\Om|\nb g|^2\e^{2s\vp_0}\,\rd x\le C\sum_{k=1}^n\int_\Om|\cL_0(\pa_kg)|^2\e^{2s\vp_0}\,\rd x,\quad\forall\,s\ge s_3.
\end{equation}
To further estimate the right-hand side of \eqref{eq-est-L0-1}, we calculate
\[\pa_k(\cL_0g)=\cL_0(\pa_kg)+\sum_{i,j=1}^n\{\pa_k(a_{ij}(\,\cdot\,,0)\pa_iu_0)\pa_jg+\pa_k\pa_i(a_{ij}(\,\cdot\,,0)\pa_ju_0)g\},\quad k=1,\ldots,n.\]
Since \eqref{eq-asp-0} also gives the boundedness of $\|a(\,\cdot\,,0)\|_{W^{2,\infty}(\Om)}$ and $\|u_0\|_{W^{3,\infty}(\Om)}$, we estimate
\[\int_\Om|\cL_0(\pa_kg)|^2\e^{2s\vp_0}\,\rd x\le\int_\Om|\pa_k(\cL_0g)|^2\e^{2s\vp_0}\,\rd x+C\int_\Om\left(|g|^2+|\nb g|^2\right)\e^{2s\vp_0}\,\rd x,\quad k=1,\ldots,n.\]
Substituting the above inequality into \eqref{eq-est-L0-1}, we obtain
\[s^2\int_\Om|\nb g|^2\e^{2s\vp_0}\,\rd x\le C\int_\Om|\nb(\cL_0g)|^2\e^{2s\vp_0}\,\rd x+C\int_\Om\left(|g|^2+|\nb g|^2\right)\e^{2s\vp_0}\,\rd x,\quad\forall\,s\ge s_3,\]
which, together with \eqref{eq-est-L0-0}, yields
\begin{align*}
s^2\int_\Om\left(|g|^2+|\nb g|^2\right)\e^{2s\vp_0}\,\rd x & \le C\int_\Om\left(|\cL_0g|^2+|\nb(\cL_0g)|^2\right)\e^{2s\vp_0}\,\rd x\\
& \quad\,+C\int_\Om\left(|g|^2+|\nb g|^2\right)\e^{2s\vp_0}\,\rd x,\quad\forall\,s\ge s_3.
\end{align*}
Then there exists a constant $s_2>s_3$ such that
\begin{equation}\label{eq-est-L0-2}
s^2\int_\Om\left(|g|^2+|\nb g|^2\right)\e^{2s\vp_0}\,\rd x\le C\int_\Om\left(|\cL_0g|^2+|\nb(\cL_0g)|^2\right)\e^{2s\vp_0}\,\rd x,\quad\forall\,s\ge s_2.
\end{equation}

Finally, it follows from $\nb(e^{s\vp_0}\cL_0g)=\e^{s\vp_0}\{\nb(\cL_0g)+s(\cL_0g)\nb\vp_0\}$ that
\[\int_\Om|\nb(\cL_0g)|^2\e^{2s\vp_0}\,\rd x\le2\int_\Om|\nb(\e^{s\vp_0}\cL_0g)|^2\,\rd x+Cs^2\int_\Om|\cL_0g|^2\e^{2s\vp_0}\,\rd x.\]
Applying the above estimate to \eqref{eq-est-L0-2}, we obtain
\[s^2\int_\Om\left(|g|^2+|\nb g|^2\right)\e^{2s\vp_0}\,\rd x\le Cs^2\|\e^{s\vp_0}\cL_0g\|_{L^2(\Om)}^2+C\int_\Om|\nb(e^{s\vp_0}\cL_0g)|^2\,\rd x,\quad\forall\,s\ge s_2.\]
Since $\cL_0$ is a first order differential operator, it reveals that $\e^{s\vp_0}\cL_0g\in H_0^1(\Om)$, which allows us to apply the Poincar\'e inequality to conclude
\begin{align*}
s^2\int_\Om\left(|g|^2+|\nb g|^2\right)\e^{2s\vp_0}\,\rd x & \le Cs^2\|\nb(\e^{s\vp_0}\cL_0g)\|_{^2(\Om)}^2+C\int_\Om|\nb(e^{s\vp_0}\cL_0g)|^2\,\rd x\\
& \le Cs^2\int_\Om|\nb(\e^{s\vp_0}\cL_0g)|^2\,\rd x,\quad\forall\,s\ge s_2.
\end{align*}
This completes the proof of Lemma \ref{lem-CE-L0}.
\end{proof}

{\bf Step 5}\ \ We complete the proof of Theorem \ref{thm-stab} in this step. By \eqref{eq-asp-obs} and the definition of $\mu_0$, we see $\mu_0=\nb\mu_0=0$ on $\pa\Om(\de_0)\supset\pa\Om\setminus\Ga$. Together with \eqref{eq-est-f0}, we obtain $\mu_0f\in H_0^2(\Om)$, which allows us to take advantage of \eqref{eq-est-L0f} and Lemma \ref{lem-CE-L0} with $g=\mu_0f$ to derive
\begin{align*}
& \quad\,\,\int_\Om\left(|\mu_0f|^2+|\nb(\mu_0f)|^2\right)\e^{2s\vp_0}\,\rd x\le C\int_\Om\left|\nb(\e^{s\vp_0}\cL_0(\mu_0f))\right|^2\,\rd x\\
& \le Cs^2\int_{Q(\de_0)}\mu_0^2\left(|f|^2+|\nb f|^2\right)\e^{2s\vp}\,\rd x\rd +C\,\e^{Cs}D^2+Cs^2\,\e^{2\eta_1s}M^2,\quad\forall\,s\gg1.
\end{align*}
Substituting
\[\int_\Om|\nb(\mu_0f)|^2\e^{2s\vp_0}\,\rd x\le2\int_\Om\mu_0^2|\nb f|^2\e^{2s\vp_0}\,\rd x+C\,\e^{2\eta_1s}M^2\]
into the above inequality, we arrive at
\begin{align*}
I_2(s) & :=\int_\Om\mu_0^2\left(|f|^2+|\nb f|^2\right)\e^{2s\vp_0}\,\rd x\\
& \le Cs^2\int_{Q(\de_0)}\mu_0^2\left(|f|^2+|\nb f|^2\right)\e^{2s\vp}\,\rd x\rd t+C\,\e^{Cs}D^2+Cs^2\,\e^{2\eta_1s}M^2,\quad\forall\,s\gg1.
\end{align*}
Owing to the choice of the weight function, we can employ Lebesgue's dominated convergence theorem to absorb the first term in the right-hand side of the above estimate in the sense that
\begin{align*}
0 & \le\lim_{s\to\infty}\f{s^2}{I_2(s)}\int_{Q(\de_0)}\mu_0^2\left(|f|^2+|\nb f|^2\right)\e^{2s\vp}\,\rd x\rd t\\
& =\lim_{s\to\infty}\f1{I_2(s)}\int_\Om\mu_0^2\left(|f|^2+|\nb f|^2\right)\e^{2s\vp_0}\int_{-T}^Ts^2\exp\left\{2s\,\e^{\la d(x)}\left(\e^{-\la\be t^2}-1\right)\right\}\rd t\rd x\\
& \le\lim_{s\to\infty}\int_{-T}^Ts^2\exp\left\{Cs\left(\e^{-\la\be t^2}-1\right)\right\}\rd t=0.
\end{align*}
Consequently, we have
\[\int_\Om\mu_0^2\left(|f|^2+|\nb f|^2\right)\e^{2s\vp_0}\,\rd x\le C\,\e^{Cs}D^2+Cs^2\,\e^{2\eta_1s}M^2,\quad\forall\,s\gg1.\]
Recalling the facts that $\de>\de_1>\de_0$ and $\mu_0\equiv1$, $\vp_0\ge\e^{\la\de}=:\eta$ in $\Om(\de)$, we further estimate the left-hand side of the above inequality from below to deduce
\begin{align*}
\e^{2\eta s}\|f\|_{H^1(\Om(\de))} & \le\int_{\Om(\de)}\left(|f|^2+|\nb f|^2\right)\e^{2s\vp_0}\,\rd x\le\int_\Om\mu_0^2\left(|f|^2+|\nb f|^2\right)\e^{2s\vp_0}\,\rd x\\
& \le C\,\e^{Cs}D^2+Cs^2\,\e^{2\eta_1s}M^2,\quad\forall\,s\gg1.
\end{align*}
Since $\eta>\eta_1$ and $s^2\le\e^{(\eta-\eta_1)s}$ for all $s\gg1$, we conclude
\[\|f\|_{\Om(\de)}\le C\,\e^{Cs}D^2+C\,\e^{-(\eta-\eta_1)s}M^2,\quad\forall\,s\gg1.\]
Finally, discussing the two cases $D\ge M$ and $D<M$ as that in \cite{JLY17a}, we eventually obtain \eqref{eq-stab-a} and finish the proof.

\Section{Iteration Method for Numerical Reconstruction}\label{sec-recon}

As the theoretical stability is guaranteed by Theorem \ref{thm-stab}, in this section we study Problem \ref{prob-cip} from the numerical viewpoint and aim at the derivation of an iteration method. Basically, the derivation is parallel to its counterpart in \cite{JLY17a}, where the corresponding inverse source problem was investigated. However, for Problem \ref{prob-cip} we shall pay special attention to the nonlinearity and the ill-posedness in the recovery of the second order coefficient.

Instead of the general governing equation \eqref{eq-gov-u} whose coefficients are all assumed to be time-dependent, throughout this section we consider the initial-boundary value problem for a hyperbolic equation with the homogeneous Neumann boundary condition
\begin{equation}\label{eq-ibvp-u}
\begin{cases}
\pa_t^2u(x,t)-\rdiv(p(x)\nb u(x,t))=F(x,t), & x\in\Om,\ 0<t<T,\\
u(x,0)=u_0(x),\quad\pa_tu(x,0)=0, & x\in\Om,\\
\pa_\nu u(x,t)=0, & x\in\pa\Om,\ 0<t<T.
\end{cases}
\end{equation}
Here the boundary condition is simplified from $p\nb u\cdot\nu=0$ on $\pa\Om\times(0,T)$ because $p$ is assumed as strictly positive on $\ov\Om$ for the non-degeneracy. We limit the numerical treatment to the time-independent case \eqref{eq-ibvp-u} not only for its simplicity, but also due to the belief that the ill-posedness are essentially the same.

For later use, we recall the classical theory on the well-posedness of \eqref{eq-ibvp-u}.

\begin{lem}[see \cite{I06,LM72}]\label{lem-ibvp-u}
Let $k=0,1,2$ and $u$ satisfy \eqref{eq-ibvp-u}, where
\[p\in W^{k,\infty}(\Om),\quad F\in H^{k-1}(\Om\times(0,T)),\quad u_0\in H^k(\Om),\]
and the $k$th order compatibility condition is satisfied on $\pa\Om\times\{0\}$. Then there exists a unique solution $u\in C([0,T];H^k(\Om))\cap C^1([0,T];H^{k-1}(\Om))$ to \eqref{eq-ibvp-u}. Moreover, there exists a constant $C>0$ depending on $p,\Om,T$ such that
\[\|u\|_{C([0,T];H^k(\Om))}+\|u\|_{C^1([0,T];H^{k-1}(\Om))}\le C\left(\|F\|_{H^{k-1}(\Om\times(0,T))}+\|u_0\|_{H^k(\Om)}\right),\quad0\le t\le T.\]
\end{lem}

Henceforth, we basically assume
\begin{equation}\label{eq-cond-F}
p\in W^{2,\infty}(\Om),\quad F\in H^1(\Om\times(0,T)),\quad u_0\in H^2(\Om),
\end{equation}
and the second order compatibility condition is satisfied on $\pa\Om\times\{0\}$. Then according to Lemma \ref{lem-ibvp-u} with $k=2$, problem \eqref{eq-ibvp-u} admits a unique solution $u\in C([0,T];H^2(\Om))\cap C^1([0,T];H^1(\Om))$. As before, we still denote the unique solution to \eqref{eq-ibvp-u} as $u(p)$.

To deal with Problem \ref{prob-cip} from the numerical aspect, we restrict ourselves to the following situation. Regarding the observation region, we consider the partial interior observation of $u(p)$ in $\om\times(0,T)$ with a subdomain $\om\subset\Om$ satisfying $\pa\om\supset\pa\Om$. In other words, we require $\om$ to cover the whole boundary $\pa\Om$, which is special in Type (II) of Problem \ref{prob-cip}. In fact, although there seems no difference between boundary and interior measurements in the theoretical stability, the latter is definitely more informative and suitable for the numerical implementation. On the other hand, it follows from Remark \ref{rem-cip} that $\pa\om\supset\pa\Om$ is a sufficient condition for the global stability of Problem \ref{prob-cip}, which is desirable for determining $p$ in the whole domain $\Om$.

In accordance with the above setting, we restrict the unknown function $p$ in
\begin{equation}\label{eq-def-adm1}
\cU_1:=\{p\in W^{2,\infty}(\Om);\,\|p\|_{H^1(\Om)}\le M_1,\ p\ge\ka_1\mbox{ in }\ov\Om\,,\ p=h_0\mbox{ on }\pa\Om\}
\end{equation}
with given constants $M_1>0,\ka_1>0$ and a given function $h_0\in W^{2,\infty}(\pa\Om)$. Compared with the admissible set $\cU$ defined in \eqref{eq-def-adm} for the theoretical stability, here we remove the restriction of $\pa_\nu p$ on $\pa\Om$. Nevertheless, we still require that $p$ is known on the whole boundary due to the key assumption \eqref{eq-asp-obs} for the stability. We refer to \cite{IY03} for the same type of admissible sets as $\cU_1$.

In practice, we are given the noisy observation data $u^\de\in L^2(\om\times(0,T))$ such that
\[\|u^\de-u(p_\true)\|_{L^2(\om\times(0,T))}\le\de,\]
where $p_\true\in\cU_1$ and $\de>0$ stand for the true solution and the noise level respectively. Now we are well prepared to recast Problem \ref{prob-cip} into a minimization problem with the Tikhonov regularization
\begin{equation}\label{eq-def-J}
\min_{p\in\cU_1}J(p),\quad J(p):=\|u(p)-u^\de\|_{L^2(\om\times(0,T))}^2+\al\|\nb p\|_{L^2(\Om)}^2,
\end{equation}
where $\al>0$ denotes the regularization parameter. Unlike the formulation in \cite{LJY15,JLY17a}, here we penalize the $L^2$-norm of $\nb p$ because one can expect certain smoothness of $p$ as the second order coefficient. Meanwhile, there is no need to penalize the $H^1$-norm of $p$ due to the boundary condition $p=h_0$ on $\pa\Om$.

As usual, we shall compute the Fr\'echet derivative of $J(p)$ in order to characterize its possible minimizer $p_*$. For arbitrarily fixed $p\in\cU_1$, we may choose any $\wt p\in W^{2,\infty}(\Om)$ such that
\begin{equation}\label{eq-cond-wtp}
\|\wt p\,\|_{W^{2,\infty}(\Om)}=1,\quad \wt p=0\mbox{ on }\pa\Om\quad\mbox{and}\quad p+\ve\,\wt p\in\cU_1
\end{equation}
holds for all sufficiently small $\ve>0$. By \eqref{eq-def-J}, we directly calculate
\begin{align}
\f{J(p+\ve\,\wt p\,)-J(p)}\ve & =\int_0^T\!\!\!\!\int_\om\f{u(p+\ve\,\wt p\,)-u(p)}\ve\left(u(p+\ve\,\wt p\,)+u(p)-2u^\de\right)\rd x\rd t\nonumber\\
& \quad\,+\al\int_\Om\nb\wt p\cdot(2\nb p+\ve\nb\wt p\,)\,\rd x,\label{eq-diff-J}
\end{align}
In order to pass $\ve\downarrow0$ in \eqref{eq-diff-J}, we need the following technical lemma.

\begin{lem}\label{lem-cov-w}
Let $u(p)$ and $u(p+\ve\,\wt p\,)$ be the solutions to \eqref{eq-ibvp-u} with coefficients $p$ and $p+\ve\,\wt p$ respectively, where $F,u_0$ satisfy \eqref{eq-cond-F}, $p\in\cU_1$ in \eqref{eq-def-adm1} and $\wt p$ satisfies \eqref{eq-cond-wtp} for all sufficiently small $\ve>0$. Then
\[\lim_{\ve\downarrow0}\|u(p+\ve\,\wt p\,)-u(p)\|_{C([0,T];H^1(\Om))}=\lim_{\ve\downarrow0}\left\|\f{u(p+\ve\,\wt p\,)-u(p)}\ve-w_0\right\|_{C([0,T];L^2(\Om))}=0,\]
where $w_0$ satisfies
\begin{equation}\label{eq-ibvp-w}
\begin{cases}
\pa_t^2w_0-\rdiv(p\nb w_0)=\rdiv(\wt p\,\nb u(p)) & \mbox{in }\Om\times(0,T),\\
w_0=\pa_tw_0=0 & \mbox{in }\Om\times\{0\},\\
\pa_\nu w_0=0 & \mbox{on }\pa\Om\times(0,T).
\end{cases}
\end{equation}
\end{lem}

\begin{proof}
Introduce $v_\ve:=u(p+\ve\,\wt p\,)-u(p)$. By taking difference of \eqref{eq-ibvp-u} with $u(p+\ve\,\wt p\,)$ and $u(p)$, it reveals that $v_\ve$ satisfies
\begin{equation}\label{eq-ibvp-v}
\begin{cases}
\pa_t^2v_\ve-\rdiv(p\nb v_\ve)=\ve\,\rdiv(\wt p\,\nb u(p+\ve\,\wt p\,)) & \mbox{in }\Om\times(0,T),\\
v_\ve=\pa_tv_\ve=0 & \mbox{in }\Om\times\{0\},\\
\pa_\nu v_\ve=0 & \mbox{on }\pa\Om\times(0,T).
\end{cases}
\end{equation}
Since $p+\ve\,\wt p$ lies in the $\ve$-neighborhood of $p$ by \eqref{eq-cond-wtp}, it follows from Lemma \ref{lem-ibvp-u} with $k=2$ that there exists a constant $C_2>0$ such that
\[\|u(p+\ve\,\wt p\,)\|_{C([0,T];H^2(\Om))}\le C_2\left(\|F\|_{H^1(\Om\times(0,T))}+\|u_0\|_{H^2(\Om)}\right)=:M_2\]
holds uniformly for all sufficiently small $\ve\ge0$. This indicates
\[\|\rdiv(\wt p\,\nb u(p+\ve\,\wt p\,))\|_{L^2(\Om\times(0,T))}\le C\|\wt p\,\|_{W^{1,\infty}(\Om)}\lim_{\ve\downarrow0}\|u(p+\ve\,\wt p\,)\|_{C([0,T];H^2(\Om))}\le CM_2\]
uniformly for all sufficiently small $\ve\ge0$, where we have $\|\wt p\,\|_{W^{1,\infty}(\Om)}\le1$ by \eqref{eq-cond-wtp}. Applying Lemma \ref{lem-ibvp-u} with $k=1$ to \eqref{eq-ibvp-v}, we obtain
\begin{align}
\lim_{\ve\downarrow0}\|u(p+\ve\,\wt p\,)-u(p)\|_{C([0,T];H^1(\Om))} & =\lim_{\ve\downarrow0}\|v_\ve\|_{C([0,T];H^1(\Om))}\nonumber\\
& \le C\lim_{\ve\downarrow0}\ve\|\rdiv(\wt p\,\nb u(p+\ve\,\wt p\,))\|_{L^2(\Om\times(0,T))}=0.\label{eq-cov-v}
\end{align}

In the same manner, we further set $w_\ve:=\ve^{-1}v_\ve$ and manipulate \eqref{eq-ibvp-w}--\eqref{eq-ibvp-v} to find
\[\begin{cases}
\pa_t^2(w_\ve-w_0)-\rdiv(p\nb(w_\ve-w_0))=\rdiv(\wt p\,\nb v_\ve) & \mbox{in }\Om\times(0,T),\\
w_\ve-w_0=\pa_t(w_\ve-w_0)=0 & \mbox{in }\Om\times\{0\},\\
\pa_\nu(w_\ve-w_0)=0 & \mbox{on }\pa\Om\times(0,T).
\end{cases}\]
Then we employ \eqref{eq-cov-v} and Lemma \ref{lem-ibvp-u} with $k=0$ to conclude
\begin{align*}
& \quad\,\,\lim_{\ve\downarrow0}\left\|\f{u(p+\ve\,\wt p\,)-u(p)}\ve-w_0\right\|_{C([0,T];L^2(\Om))}=\lim_{\ve\downarrow0}\|w_\ve-w_0\|_{C([0,T];L^2(\Om))}\\
& \le C\lim_{\ve\downarrow0}\|\rdiv(\wt p\,\nb v_\ve)\|_{H^{-1}(\Om\times(0,T))}\le C\|\wt p\,\|_{L^\infty(\Om)}\lim_{\ve\downarrow0}\|v_\ve\|_{C([0,T];H^1(\Om))}=0,
\end{align*}
which finishes the proof.
\end{proof}

Now that the convergence is guaranteed by the above lemma, we can pass $\ve\downarrow0$ in \eqref{eq-diff-J} to deduce
\begin{align}
\f{J'(p)\wt p}2 & =\lim_{\ve\downarrow0}\f{J(p+\ve\,\wt p\,)-J(p)}{2\ve}=\int_0^T\!\!\!\!\int_\om w_0\left(u(p)-u^\de\right)\rd x\rd t+\al\int_\Om\nb p\cdot\nb\wt p\,\rd x\rd t\nonumber\\
& =\int_0^T\!\!\!\!\int_\Om w_0\,\chi_\om\left(u(p)-u^\de\right)\rd x\rd t-\al\int_\Om \wt p\,\tri p\,\rd x,\label{eq-div-J}
\end{align}
where $\chi_\om$ denotes the characteristic function of $\om$, and we utilized $\wt p=0$ on $\pa\Om$ to obtain \eqref{eq-div-J} by integration by parts.

In order to derive the explicit form of $J'(p)$, we should further transform the first term on the right-hand side of \eqref{eq-div-J}. To this end, we follow the same line as that in \cite{LJY15,JLY17a} to introduce the backward problem
\begin{equation}\label{eq-ibvp-z}
\begin{cases}
\pa_t^2z-\rdiv(p\nb z)=\chi_\om\left(u(p)-u^\de\right) & \mbox{in }\Om\times(0,T),\\
z=\pa_tz=0 & \mbox{in }\Om\times\{T\},\\
\pa_\nu z=0 & \mbox{on }\pa\Om\times(0,T).
\end{cases}
\end{equation}
To clarify the dependency, we also denote the solution to \eqref{eq-ibvp-z} as $z(p)$. Since $\chi_\om\left(u(p)-u^\de\right)\in L^2(\Om\times(0,T))$, Lemma \ref{lem-ibvp-u} gives $z(p)\in H^1(\Om\times(0,T))$. On the other hand, it can be inferred from the proof of Lemma \ref{lem-cov-w} that the solution of \eqref{eq-ibvp-w} satisfies $w_0\in H^1(\Om\times(0,T))$ and $w_0|_{t=0}=0$. Hence, in view of the weak solution of hyperbolic equations, we can regard $w_0$ and $z(p)$ as mutual test functions of each other, so that we can further treat
\begin{align*}
\int_0^T\!\!\!\!\int_\Om w_0\,\chi_\om\left(u(p)-u^\de\right)\rd x\rd t & =\int_0^T\!\!\!\!\int_\Om\left(p\nb z(p)\cdot\nb w_0-(\pa_tz(p))\pa_tw_0\right)\rd x\rd t\\
& =\int_0^T\!\!\!\!\int_\Om z(p)\,\rdiv(\wt p\,\nb u(p))\,\rd x\rd t=-\int_0^T\!\!\!\!\int_\Om \wt p\,\nb u(p)\cdot\nb z(p)\,\rd x\rd t.
\end{align*}
Substituting the above identity into \eqref{eq-div-J}, we arrive at
\[\f{J'(p)\wt p}2=-\int_\Om\left(\int_0^T\nb u(p)\cdot\nb z(p)\,\rd t+\al\,\tri p\right)\wt p\,\rd x.\]
Since $\wt p$ was taken arbitrarily which satisfies \eqref{eq-cond-wtp}, this suggests a characterization of the minimizer to the problem \eqref{eq-def-J}.

\begin{prop}\label{prop-min}
Let $\cU_1$ be the admissible set defined in \eqref{eq-def-adm1}, and $J(p)$ be the functional defined in \eqref{eq-def-J}. Then $p_*\in\cU_1$ is a minimizer of $J(p)$ within $\cU_1$ only if it satisfies the variational equation
\begin{equation}\label{eq-var-p*}
\int_0^T\nb u(p_*)\cdot\nb z(p_*)\,\rd t+\al\,\tri p_*=0,
\end{equation}
where $u(p_*)$ and $z(p_*)$ solve the forward system \eqref{eq-ibvp-u} and the backward one \eqref{eq-ibvp-z} with the coefficient $p_*$, respectively.
\end{prop}

On the basis of \eqref{eq-var-p*}, we design the following iteration scheme
\begin{equation}\label{eq-itr}
\left\{\!\begin{alignedat}{2}
& \tri p_{m+1}=\f K{K+\al}\tri p_m-\f1{K+\al}\int_0^T\nb u(p_m)\cdot\nb z(p_m)\,\rd t & \quad & \mbox{in }\Om,\\
& p_{m+1}=h_0 & \quad & \mbox{on }\pa\Om,
\end{alignedat}\right.\quad m=0,1,\ldots,
\end{equation}
where $K>0$ is a tuning parameter. In other words, given the result $p_m$ of the previous step, we have to solve the forward system \eqref{eq-ibvp-u}, the backward system \eqref{eq-ibvp-z} and the boundary value problem \eqref{eq-itr} for a Poisson equation subsequently to obtain $p_{m+1}$. In comparison with the inverse source problems treated in \cite{LJY15,JLY17a,JLY17b,JLLY17}, we see that both solutions to forward and backward problems appear in \eqref{eq-itr} due to the nonlinearity of Problem \ref{prob-cip}. More importantly, here we should update $p_m$ indirectly by solving an extra Poisson equation since we penalize $\nb p$ instead of $p$ itself. Such an additional procedure, however, does not affect the efficiency because the computational cost for solving \eqref{eq-itr} is rather minor compared with that for solving two time evolution equations. On the other hand, in view of the variational principle, it is readily seen that the solution $p_{m+1}$ of \eqref{eq-itr} coincides with the minimizer of the minimization problem
\begin{equation}\label{eq-min}
\min\int_\Om\left\{\f12|\nb p|^2+p\left(\f K{K+\al}\tri p_m-\f1{K+\al}\int_0^T\nb u(p_m)\cdot\nb z(p_m)\,\rd t\right)\right\}\rd x
\end{equation}
for all $p\in H^1(\Om)$ satisfying $p=h_0$ on $\pa\Om$.

Concerning the convergence issue, we notice the relation between the iteration \eqref{eq-itr} and the minimization problem of a surrogate functional
\begin{equation}\label{eq-def-Js}
J^s(p,q):=J(p)+K\|\nb(p-q)\|_{L^2(\Om)}^2-\|u(p)-u(q)\|_{L^2(\om\times(0,T))}^2,\quad p,q\in\cU_1.
\end{equation}
Indeed, let us fix $q$ and consider the minimization of $J^s(p,q)$ with respect to $p$ which is sufficiently close to $q$, e.g., $\|p-q\|_{W^{1,\infty}(\Om)}\ll1$. Separating the terms involving $p$ from others in \eqref{eq-def-Js}, we treat $J^s(p,q)$ as
\begin{align}
J^s(p,q) & =(K+\al)\|\nb p\|_{L^2(\Om)}^2-2K\int_\Om\nb p\cdot\nb q\,\rd x+2\int_0^T\!\!\!\!\int_\om u(p)\left(u(q)-u^\de\right)\rd x\rd t\nonumber\\
& \quad\,+\|u^\de\|_{L^2(\om\times(0,T))}^2-\|u(q)\|_{L^2(\om\times(0,T))}^2+K\|\nb q\|_{L^2(\Om)}^2\nonumber\\
& =(K+\al)\|\nb p\|_{L^2(\Om)}^2+2K\int_\Om p\,\tri q\,\rd x+2\int_0^T\!\!\!\!\int_\Om v\,\chi_\om\left(u(q)-u^\de\right)\rd x\rd t\nonumber\\
& \quad\,+\|u(q)-u^\de\|_{L^2(\om\times(0,T))}+K\left(\|\nb q\|_{L^2(\Om)}^2-2\int_{\pa\Om}h_0\,\pa_\nu q\,\rd\si\right)\label{eq-Js-1},
\end{align}
where $v:=u(p)-u(q)$ satisfies
\[\begin{cases}
\pa_t^2v-\rdiv(p\nb v)=\rdiv((p-q)\nb u(q)) & \mbox{in }\Om\times(0,T),\\
v=\pa_tv=0 & \mbox{in }\Om\times\{0\},\\
\pa_\nu v=0 & \mbox{on }\pa\Om\times(0,T).
\end{cases}\]
Utilizing the backward problem \eqref{eq-ibvp-z}, we take $z(q)$ and $v$ as mutual test functions to deduce
\begin{align}
& \quad\,\,\int_0^T\!\!\!\!\int_\Om v\,\chi_\om\left(u(q)-u^\de\right)\rd x\rd t=\int_0^T\!\!\!\!\int_\Om(q\nb v\cdot\nb z(q)-(\pa_tv)\,\pa_tz(q))\,\rd x\rd t\nonumber\\
& =\int_0^T\!\!\!\!\int_\Om(p\nb v\cdot\nb z(q)-(\pa_tv)\,\pa_tz(q))\,\rd x\rd t+\int_0^T\!\!\!\!\int_\Om(q-p)\nb v\cdot\nb z(q)\,\rd x\rd t\nonumber\\
& =\int_0^T\!\!\!\!\int_\Om\rdiv((p-q)\nb u(q))\,z(q)\,\rd x\rd t+O(\|p-q\|_{W^{1,\infty}(\Om)}^2)\nonumber\\
& =-\int_0^T\!\!\!\!\int_\Om(p-q)\nb u(q)\cdot\nb z(q)\,\rd x\rd t+O(\|p-q\|_{W^{1,\infty}(\Om)}^2),\label{eq-Js-2}
\end{align}
where we used $p-q=0$ on $\pa\Om$ and applied Lemma \ref{lem-ibvp-u} with $k=1$ to $v$ to estimate
\begin{align*}
\|v\|_{C([0,T];H^1(\Om))} & \le C\|\rdiv((p-q)\nb u(q))\|_{L^2(\Om\times(0,T))}\\
& \le C\|u(q)\|_{C([0,T];H^2(\Om))}\|p-q\|_{W^{1,\infty}(\Om)}\le C\|p-q\|_{W^{1,\infty}(\Om)}
\end{align*}
within the admissible set $\cU_1$. Substituting \eqref{eq-Js-2} into \eqref{eq-Js-1}, we collect the constant component of $J^s(p,q)$ as $C_3$ and conclude
\begin{align*}
J^s(p,q) & =2(K+\al)\int_\Om\left\{\f12|\nb p|^2+p\left(\f K{K+\al}\tri q-\f1{K+\al}\int_0^T\nb u(q)\cdot\nb z(q)\,\rd t\right)\right\}\rd x\\
& \quad\,+O(\|p-q\|_{W^{1,\infty}(\Om)}^2)+C_3.
\end{align*}
Comparing the above expression with \eqref{eq-min}, we figure out that the iterative update \eqref{eq-itr} is almost equivalent to solving a series of minimization problem
\begin{equation}\label{eq-min-sur}
\min_{p\in\cU_1}J^s(p,p_m),\quad m=0,1,\ldots,
\end{equation}
provided that $\|p-p_m\|_{W^{1,\infty}(\Om)}$ is sufficiently small. On the other hand, it is well known that the convergence of \eqref{eq-min-sur} is guaranteed by the positivity of the surrogate functional $J^s(p,q)$ for all $p,q\in\cU_1$ (see \cite{RT06}). By definition, this is achieved by taking sufficiently large $K>0$ such that
\begin{equation}\label{eq-choice-K}
\|u(p)-u(q)\|_{L^2(\om\times(0,T))}^2\le K\|\nb(p-q)\|_{L^2(\Om)}^2,\quad\forall\,p,q\in\cU_1.
\end{equation}
Consequently, it reveals that the convergence of \eqref{eq-min-sur} almost indicates the convergence of our proposed iteration \eqref{eq-itr}. Unfortunately, due to the nonlinearity of the problem, we cannot remove the second order term $O(\|p-q\|_{W^{1,\infty}(\Om)}^2)$ in \eqref{eq-Js-2}, which prevents us from proving the convergence rigorously.

We close this section by summarizing the main algorithm for the numerical reconstruction.

\begin{algo}\label{algo-itr}
Fix the boundary value $h_0$ of $p_\true$. Choose a tolerance $\ep>0$, a regularization parameter $\al>0$ and a suitably large tuning constant $K>0$. Give an initial guess $p_0$ and set $m=0$.
\begin{enumerate}
\item Compute $p_{m+1}$ according to the iterative update \eqref{eq-itr}.
\item If $\|p_{m+1}-p_m\|_{L^2(\Om)}/\|p_m\|_{L^2(\Om)}\le\ep$, then stop the iteration. Otherwise, update $m\leftarrow m+1$ and return to Step 1.
\end{enumerate}
\end{algo}

\Section{Numerical Examples}\label{sec-numer}

In this section, we apply the iteration method proposed in the previous section to the numerical treatment for Problem \ref{prob-cip}, and evaluate its numerical performance. More precisely, we shall implement Algorithm \ref{algo-itr} to reconstruct the principal coefficient $p$ in the hyperbolic equation \eqref{eq-ibvp-u}. As the first attempt, we restrict ourselves to one spatial dimension, and simply set $\Om=(0,1)$ and $T=1$. We divide $\ov\Om\times[0,T]=[0,1]^2$ into $100\times100$ equidistant meshes, and employ some unconditionally finite difference methods to solve the $3$ equations involved in Algorithm \ref{algo-itr}, namely, \eqref{eq-ibvp-u}, \eqref{eq-ibvp-z} and \eqref{eq-itr}.

We specify various coefficients and parameters to be used in the numerical tests as follows. For the source term $F$ and the initial value $u_0$ of \eqref{eq-ibvp-u}, we fix
\[F(x,t)=x+t+1,\quad u_0(x)\equiv1.\]
For the boundary condition of $p$ required in \eqref{eq-def-adm1}, we simply set
\[p|_{\pa\Om}=h_0\equiv1\quad\mbox{on }\pa\Om=\{0,1\}.\]
Given the true solution $p_\true$ and thus the noiseless data $u(p_\true)$, we generate the noisy data $u^\de$ by adding uniform random noises in such a way that
\[u^\de(x,t)=u(p_\true)(x,t)+\de\,\rand(-1,1),\quad x\in\om,\ 0<t<T,\]
where $\rand(-1,1)$ denotes the random number uniformly distributed in $[-1,1]$. For the noise level $\de>0$, we choose it as a certain portion of the amplitude of the noiseless data, that is,
\[\de:=\de_0\|u(p_\true)\|_{C(\ov\Om\times[0,T])},\quad0<\de_0<1.\]
For the tuning parameter $K$, it should be chosen sufficiently large to guarantee the convergence (see \eqref{eq-choice-K}). Roughly speaking, it depends on the operator norm of the forward operator which maps $p$ to $u(p)|_{\om\times(0,T)}$, which is impossible to compute in practice. Hence, we have to postulate that $K$ is proportional to the size $|\om|$ of the observation subdomain. Analogously, we also make empirical choices of the the regularization parameter $\al$ and the stopping criteria $\ep$ in Algorithm \ref{algo-itr} in such a way that
\begin{equation}\label{eq-para}
K\propto|\om|,\quad \al\propto\de,\quad\ep\propto\de_0.
\end{equation}
In all examples, we fix the initial guess as $p_0\equiv1$. As usual, we evaluate the numerical performance of Algorithm \ref{algo-itr} by the number $N$ of iterations, the relative $L^2$ error
\[\err:=\f{\|p_N-p_\true\|_{L^2(\Om)}}{\|p_\true\|_{L^2(\Om)}},\]
the elapsed time and sometimes the illustrative figures, and we recognize $p_N$ as the result of the numerical reconstruction.

\begin{ex}\label{ex-1}
First, we test Algorithm \ref{algo-itr} with several choices of true solutions $p_\true$ to demonstrate its accuracy and efficiency. More precisely, we fix the subdomain $\om=\Om\setminus[0.1,0.9]$ and the relative noise level $\de_0=1\%$. Correspondingly, we choose $K=2\times10^{-5}$ and $\al=10^{-7}$. The following true solutions with different shapes and smoothness are taken into consideration:
\begin{enumerate}
\item[(a)] A smooth and symmetric true solution $p_\true(x)=\f12\sin\pi x+1$.
\item[(b)] An asymmetric true solution $p_\true(x)=x(x-1)(x-\f32)+1$.
\item[(c)] A non-smooth true solution $p_\true(x)=\f12\min(x,1-x)+\f14\sin\pi x+1$.
\end{enumerate}
Various aspects of the numerical performance are listed in Table \ref{tab-ex1}. The comparisons of true solutions with their reconstructed ones are illustrated in Figure \ref{fig-ex1}.
\begin{table}[htbp]\centering
\caption{The numerical performance of Algorithm \ref{algo-itr} for various choices of true solutions in Example \ref{ex-1}.}\label{tab-ex1}
\begin{tabular}{ll|lll}
\hline\hline
Case & $p_\true(x)$ & $N$ & $\err$ & Elapsed time\\
\hline
(a) & $\f12\sin\pi x+1$ & $21$ & $0.86\%$ & $0.51\,\s$\\
(b) & $x(x-1)(x-\f32)+1$ & $35$ & $0.61\%$ & $0.97\,\s$\\
(c) & $\f12\min(x,1-x)+\f14\sin\pi x+1$ & $9$ & $1.72\%$ & $0.25\,\s$\\
\hline\hline
\end{tabular}
\end{table}
\begin{figure}[htbp]\centering
\subfigure[$p_\true(x)=\f12\sin\pi x+1$]{
\includegraphics[trim=5mm 2mm 12mm 5mm,clip=true,width=.48\textwidth]{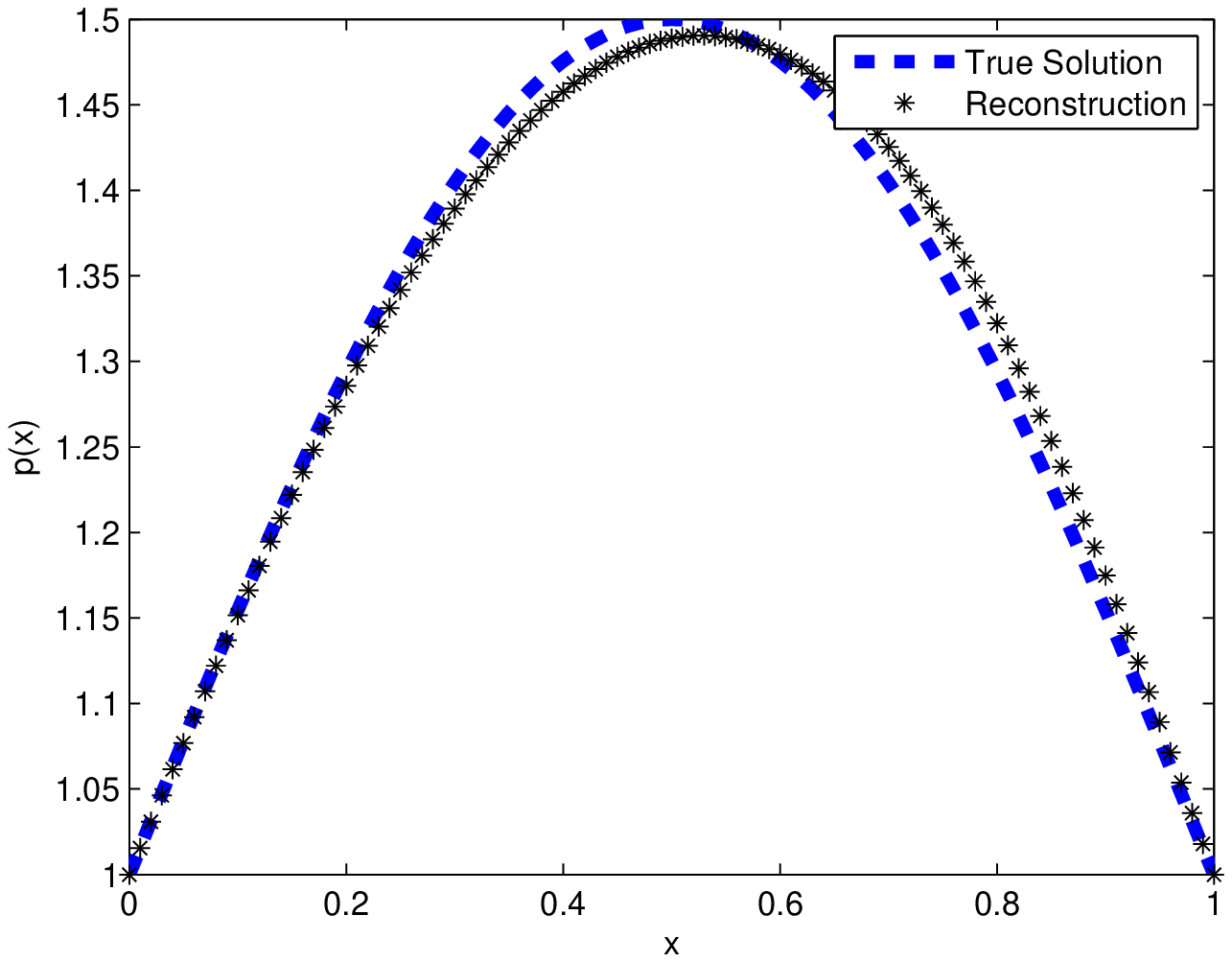}}
\subfigure[$p_\true(x)=x(x-1)(x-\f32)+1$]{
\includegraphics[trim=5mm 2mm 12mm 5mm,clip=true,width=.48\textwidth]{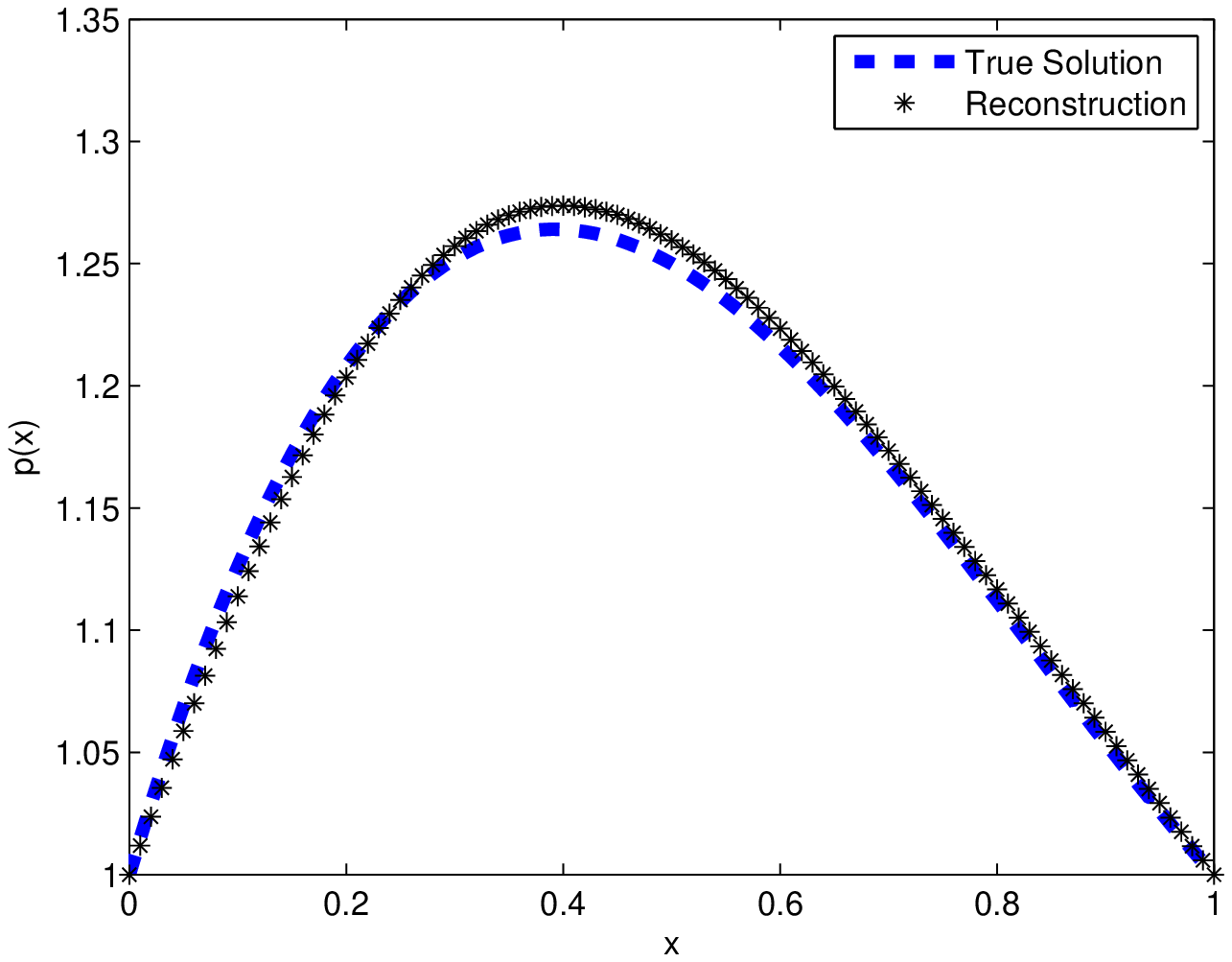}}
\subfigure[$p_\true(x)=\f12\min(x,1-x)+\f14\sin\pi x+1$]{
\includegraphics[trim=5mm 2mm 12mm 0mm,clip=true,width=.48\textwidth]{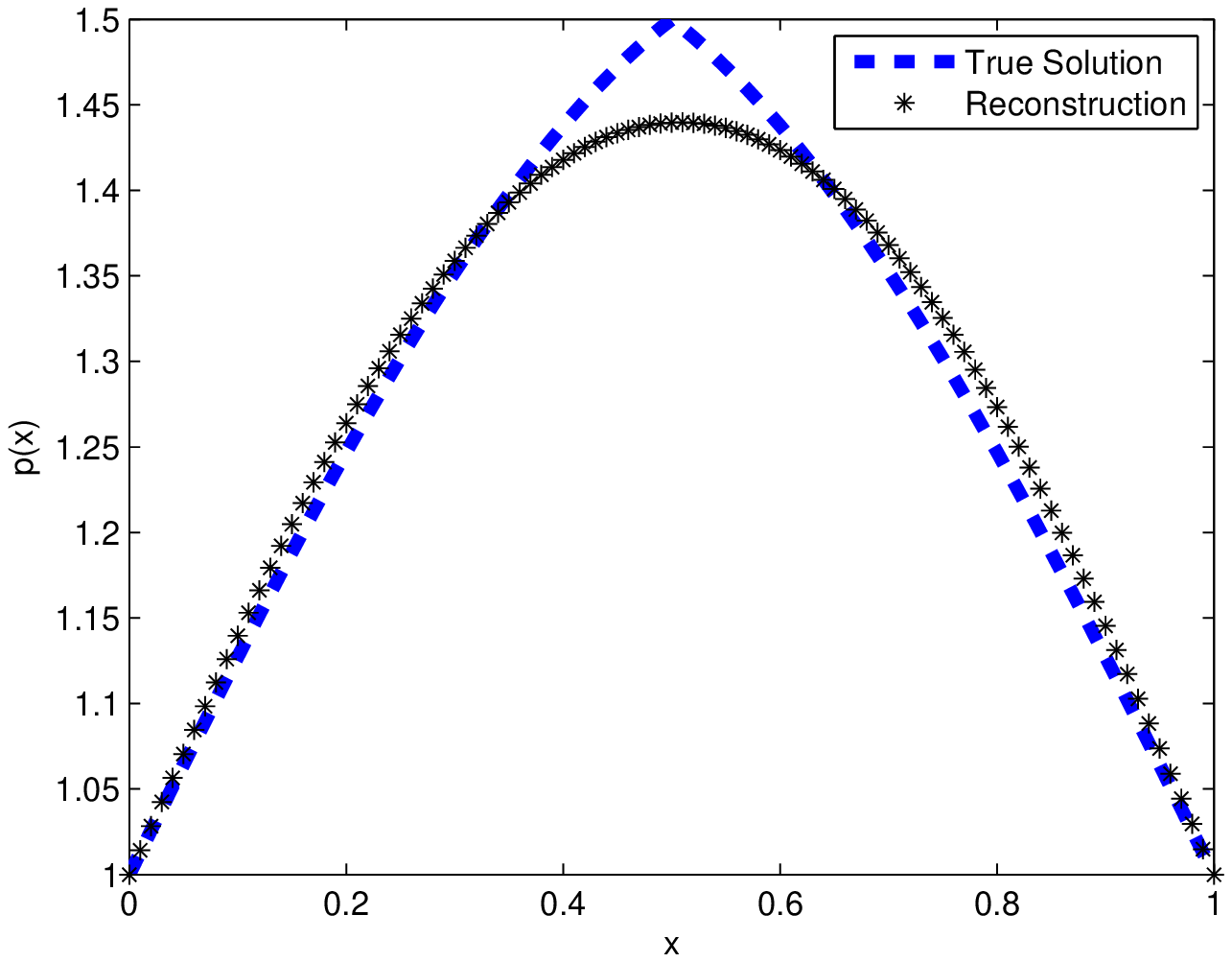}}
\caption{True solutions and the corresponding reconstructions in Example \ref{ex-1}.}\label{fig-ex1}
\end{figure}
\end{ex}

\begin{ex}\label{ex-2}
In this example, we fix the true solution as
\[p_\true(x)=\f12\sin\pi x+1\]
and evaluate the performance of Algorithm \ref{algo-itr} with different combinations of noise levels and observation subdomains. In detail, we first fix the relative noise level as $\de_0=1\%$ as that in Example \ref{ex-1}, and change the observation subdomain $\om$ as
\[\om=\Om\setminus[0.2,0.8],\quad\om=\Om\setminus[0.1,0.9],\quad\om=\Om\setminus[0.05,0.95]\]
with decreasing sizes. Next, we fix $\om=\Om\setminus[0.1,0.9]$ and increase the relative noises $\de_0$ as $0\%$, $1\%$, $2\%$, $4\%$ and $8\%$. In accordance with the above combinations of $\de_0$ and $\om$, we also change the parameters $M$ and $\al$ according to \eqref{eq-para}. The choices of parameters in the tests and the resulting numerical performance are listed in Table \ref{tab-ex2}.
\begin{table}[htbp]\centering
\caption{Parameters and corresponding numerical performance in Example \ref{ex-2} under various combinations of the relative noise levels $\de_0$ and the observation subdomains $\om$.}\label{tab-ex2}
\begin{tabular}{ll|ll|lll}
\hline\hline
$\om$ & $\de_0$ & $K$ & $\al$ & $N$ & $\err$ & Elapsed time\\
\hline
$\Om\setminus[0.2,0.8]$ & $1\%$ & $4\times10^{-5}$ & $10^{-7}$ & $17$ & $0.66\%$ & $0.43\,\s$\\
$\Om\setminus[0.1,0.9]$ & $1\%$ & $2\times10^{-5}$ & $10^{-7}$ & $21$ & $0.86\%$ & $0.51\,\s$\\
$\Om\setminus[0.05,0.95]$ & $1\%$ & $10^{-5}$ & $10^{-7}$ & $39$ & $0.94\%$ & $0.80\,\s$\\
\hline
$\Om\setminus[0.1,0.9]$ & $0\%$ & $2\times10^{-5}$ & $10^{-9}$ & $9$ & $0.54\%$ & $0.31\,\s$\\
$\Om\setminus[0.1,0.9]$ & $2\%$ & $2\times10^{-5}$ & $2\times10^{-7}$ & $33$ & $1.39\%$ & $0.96\,\s$\\
$\Om\setminus[0.1,0.9]$ & $4\%$ & $2\times10^{-5}$ & $4\times10^{-7}$ & $128$ & $3.06\%$ & $2.98\,\s$\\
$\Om\setminus[0.1,0.9]$ & $8\%$ & $2\times10^{-5}$ & $8\times10^{-7}$ & $121$ & $18.70\%$ & $2.23\,\s$\\
\hline\hline
\end{tabular}
\end{table}
\end{ex}

The above examples demonstrate the accuracy and robustness of Algorithm \ref{algo-itr} as its previous applications e.g.\! in \cite{LJY15,JLY17a}. Especially, in the reconstruction of the principal coefficient, it is obvious that the problem suffers from stronger ill-posedness and nonlinearity compared with the corresponding inverse source problem. Even though, the proposed method still provides satisfactory results with rather small observation subdomain and moderately large noise in observation data.

In Example \ref{ex-1}, our method proves its feasibility for various choices of true solutions, even including a non-smooth one. This suggests the possibility of relaxing the assumption $p_\true\in W^{2,\infty}(\Om)$ in the derivation of Algorithm \ref{algo-itr}. Unfortunately, it is shown in Figure \ref{fig-ex1}(c) that our method fails to capture the local non-smoothness far away from $\om$ because of its $L^2$-based formulation.

On the other hand, Example \ref{ex-2} illustrates the influence of the size of $\om$ and the noise level upon the numerical performance. The results agree well with our common sense, namely, a smaller observation subdomain $\om$ results in a larger relative error with more iteration steps until convergence. Still, we need the coverage $\pa\om\supset\pa\Om$ for the numerical stability. Meanwhile, both error and iteration steps also increase with larger noise level as expected. However, we see in Table \ref{tab-ex2} that an $8\%$ relative noise causes considerably large error in the reconstruction, possibly due to the strong nonlinearity of the problem.

\Section{Concluding Remarks}\label{sec-remark}

The propose of this article is to investigate Problem \ref{prob-cip}, namely, the determination of the spatial component $p(x)$ in the second order coefficient of a hyperbolic equation, from both theoretical and numerical aspects. Theoretically, we are mainly motivated by the existing literature represented by \cite{IY03} and formulate the problem within the general hyperbolic operator $\cH_p$ with a time-dependent principal part. On the same direction of \cite{JLY17a}, we take advantage of the key Carleman estimates for $\cH_p$ to establish a local H\"older stability result for Problem \ref{prob-cip}. The proof starts from the routine linearization, but unlike \cite{IY03} we should turn to another Carleman estimate to dominate the $H^1$-norm of $p-q$ by that of $\rdiv((p-q)a\nb u(q))$. The reason traces back to our choice of including $p$ in the divergence in \eqref{eq-def-Hp} for a concrete physical meaning. Instead, if we base the discussion on a nearly non-divergence form
\[\pa_t^2u-p\,\rdiv(a\nb u)-b\cdot\nb u+c\,u=F,\]
then the source term after linearization becomes $(p-q)\rdiv(a\nb u(q))$, and the $L^2$ estimate of $p-q$ reduces to an immediate corollary of \cite[Theorem 2.3]{JLY17a}. The same comment applies to the determination of any spatial components of lower order coefficients in $\cH_p$, by which we can expect the identical stability result. In these cases, it suffices to replace \eqref{eq-asp-1} by some analogous non-vanishing assumptions, and we omit the details here.

In the numerical aspect, we adopt the orthodox Tikhonov regularization to interpret Problem \ref{prob-cip} as a minimization problem. For the highest order coefficient $p$, we penalize the $L^2$-norm of $\nb p$ with its information given on the whole boundary. Calculating the Fr\'echet derivative, we derive the variational equation for a minimizer of the Tikhonov functional, which involves a backward problem and the Laplacian of $p$. This suggests a novel iterative update \eqref{eq-itr}, where one should solve a Poisson equation at each step. Moreover, by the variational principle we find a link between \eqref{eq-itr} and the minimization of a corresponding surrogate functional. Unfortunately, the convergence of the latter does not imply that of the former, because their equivalence is not rigorous due to the nonlinearity of Problem \ref{prob-cip}.

We conclude this paper with some possible future topics related to Problem \ref{prob-cip}. As was mentioned in Remark \ref{rem-cip}, the local stability in Theorem \ref{thm-stab} relies heavily on the choice of the weight function $\vp$ in Carleman estimates. We shall consider the possibility of a clever choice of $\vp$ which optimizes the stability and reduces the observation cost. Meanwhile, another interesting issue is the simultaneous determination of several coefficients, e.g., finding $p,q$ in
\[\pa_t^2u(x,t)-\rdiv(\mathrm{diag}(p(x),q(x))\nb u(x,t))=F(x,t).\]
For such kind of problems, possibly one should take several measurements. As a similar but far more difficult case, one can study the same problem for linear anisotropic Lam\'e systems with time-dependent principal parts. Numerically, the idea of solving an auxiliary equation seems fresh in iteration methods to the best of our knowledge. We are interested in applying it to other inverse problems and analyze its properties, especially convergence.

\bigskip

{\bf Acknowledgement}\ \ The authors appreciate the valuable discussions with Jin Cheng and Shuai Lu (Fudan University). This work is supported by A3 Foresight Program ``Modeling and Computation of Applied Inverse Problems'', Japan Society for the Promotion of Science (JSPS) and National Natural Science Foundation of China. The second and the third authors are partially supported by Grant-in-Aid for Scientific Research (S) 15H05740, JSPS. The second author is supported by JSPS Postdoctoral Fellowship for Overseas Researchers and Grant-in-Aid for JSPS Fellows 16F16319.


\begin{thebibliography}{99}

\bibitem{A75}
R. A. Adams, {\it Sobolev Spaces}, Academic Press, New York, 1975.

\bibitem{B04}
M. Bellassoued, Uniqueness and stability in determining the speed of propagation of second-order hyperbolic equation with variable coefficients, {\it Appl. Anal.}, {\bf83} (2004), 983--1014.

\bibitem{BY06}
M. Bellassoued and M. Yamamoto, Logarithmic stability in determination of a coefficient in an acoustic equation by arbitrary boundary observation, {\it J. Math. Pures Appl.}, {\bf85} (2006), 193--224.

\bibitem{BY08}
M. Bellassoued and M. Yamamoto, Determination of a coefficient in the wave equation with a single measurement, {\it Appl. Anal.} {\bf87} (2008), 901--920.

\bibitem{BK81}
A. L. Bukhgeim and M. V. Klibanov, Global uniqueness of a class of multidimensional inverse problems, {\it Soviet Math. Dokl.}, {\bf24} (1981), 244--247.

\bibitem{DDM04}
I. Daubechies, M. Defrise and C. De Mol, An iterative thresholding algorithm for linear inverse problems with a sparsity constraint, {\it Comm. Pure Appl. Math.}, {\bf57} (2004), 1413--1457.

\bibitem{DTV07}
I. Daubechies, G. Teschke and L. Vese, Iteratively solving linear inverse problems under general convex constraints, {\it Inverse Probl. Imaging}, {\bf1} (2007), 29--46.

\bibitem{IY01a}
O. Y. Imanuvilov and M. Yamamoto, Global uniqueness and stability in determining coefficients of wave equations, {\it Comm. Partial Differential Equations}, {\bf26} (2001), 1409--1425.

\bibitem{IY01b}
O. Y. Imanuvilov and M. Yamamoto, Global Lipschitz stability in an inverse hyperbolic problem by interior observation, {\it Inverse Problems}, {\bf17} (2001), 717--728.

\bibitem{IY03}
O. Y. Imanuvilov and M. Yamamoto, Determination of a coefficient in an acoustic equation with a single measurement, {\it Inverse Problems}, {\bf19} (2003), 157--171.

\bibitem{I93}
V. Isakov, Uniqueness and stability in multi-dimensional inverse problems, {\it Inverse Problems}, {\bf9} (1993), 579--621.

\bibitem{I06}
V. Isakov, {\it Inverse Problems for Partial Differential Equations}, Springer, New York, 2006.

\bibitem{JFZ14}
D. Jiang, H. Feng and J. Zou, Overlapping domain decomposition methods for linear inverse problems, {\it Inverse Probl. Imaging}, {\bf9} (2014), 163--188.

\bibitem{JLLY17}
D. Jiang, Z. Li, Y. Liu and M. Yamamoto, Weak unique continuation property and a related inverse source problem for time-fractional diffusion-advection equations, {\it Inverse Problems}, {\bf33} 2017, 055013.

\bibitem{JLY17a}
D. Jiang, Y. Liu and M. Yamamoto, Inverse source problem for the hyperbolic equation with a time-dependent principal part, {\it J. Differential Equations}, {\bf262} (2017), 653--681.

\bibitem{JLY17b}
D. Jiang, Y. Liu and M. Yamamoto, Inverse source problem for a wave equation with final observation data, {\it Mathematical Analysis of Continuum Mechanics and Industrial Applications}, H. Itou et al. (eds.), Springer, Singapore (2017), 153--164.

\bibitem{K87}
A. Kha\u\i darov, Carleman estimates and inverse problems for second order hyperbolic equations, {\it Math. USSR Sbornik}, {\bf58} (1987), 267--277.

\bibitem{K92}
M. V. Klibanov, Inverse problems and Carleman estimates, {\it Inverse problems}, {\bf8} (1992), 575--596.

\bibitem{LM72}
J.-L. Lions and E. Magenes, {\it Non-homogeneous Boundary Value Problems and Applications}, Springer, Berlin, 1972.

\bibitem{LJY15}
Y. Liu, D. Jiang and M. Yamamoto, Inverse source problem for a double hyperbolic equation describing the three-dimensional time cone model, {\it SIAM J. Appl. Math.}, {\bf75} (2015), 2610--2635.

\bibitem{LXY12}
Y. Liu, X. Xu and M. Yamamoto, Growth rate modeling and identification in the crystallization of polymers, {\it Inverse Problems}, {\bf28} (2012), 095008.

\bibitem{LY14}
Y. Liu and M. Yamamoto, On the multiple hyperbolic systems modelling phase transformation kinetics, {\it Appl. Anal.}, {\bf93} (2014), 1297--1318.

\bibitem{RT06}
R. Ramlau and G. Teschke, A Tikhonov-based projection iteration for nonlinear ill-posed problems with sparsity constraints, {\it Numer. Math.}, {\bf104} (2006), 177--203.

\bibitem{Y99}
M. Yamamoto, Uniqueness and stability in multidimensional hyperbolic inverse problems, {\it J. Math. Pures Appl.}, {\bf78} (1999), 65--98.

\end{thebibliography}
\end{document}